\documentclass[11pt]{article}
\usepackage{graphicx}
\usepackage{amsfonts}
\usepackage{texdraw}
\usepackage{amssymb, amsmath, amsthm}
\usepackage{epsfig}
\usepackage[english]{babel}
\usepackage{latexsym}
\usepackage{amscd}

\newtheorem{theorem}{Theorem}

\newtheorem{definition}{Definition}
\newtheorem{lemma}{Lemma}
\newtheorem{proposition}{Proposition}

\newtheorem{remark}{Remark}

\def\neweq#1{\begin{equation}\label{#1}}
\def\endeq{\end{equation}}
\def\eq#1{(\ref{#1})}
\def\proof{\noindent{\it Proof. }}
\def\endproof{\hfill $\Box$\par\vskip3mm}

\newcommand{\R}{\mathbb{R}}
\newcommand{\N}{\mathbb{N}}

\newcommand{\eps}{\varepsilon}

\renewcommand{\arraystretch}{1.5}
\begin{document}

\title{A qualitative explanation of the origin of torsional instability\\
in suspension bridges}

\author{Elvise BERCHIO $^\sharp$ - Filippo GAZZOLA $^\dagger$}
\date{}
\maketitle

\begin{center}
{\small $^\sharp$ Dipartimento di Scienze Matematiche-Politecnico di Torino-Corso Duca degli Abruzzi 24-10129 Torino, Italy\\
        $^\dagger$ Dipartimento di Matematica - Politecnico di Milano - Piazza Leonardo da Vinci 32 - 20133 Milano, Italy\\
{\tt elvise.berchio@polito.it, filippo.gazzola@polimi.it}}
\end{center}

\begin{abstract}
We consider a mathematical model for the study of the dynamical behavior of suspension bridges. We show that internal resonances, which depend on the bridge structure
only, are the origin of torsional instability. We obtain both theoretical and numerical estimates of the thresholds of instability. Our method is based on a finite
dimensional projection of the phase space  which reduces the stability analysis of the model to the stability of suitable Hill equations.
This gives an answer to a long-standing question about the origin of torsional instability in suspension bridges.\par\noindent
\end{abstract}

\section{Introduction}

The collapse of the Tacoma Narrows Bridge, which occurred in 1940, raised many questions about the stability of suspension bridges.
In particular, since the Federal Report \cite{ammann} considers {\em the crucial event in the collapse to be the sudden change from a vertical to a
torsional mode of oscillation}, see also \cite{scott}, a natural question appears to be:
$$
\mbox{{\bf why do torsional oscillations appear suddenly in suspension bridges?}}\eqno{\rm{{\bf (Q)}}}
$$
The main purpose of the present paper is to give an answer to {\bf (Q)} by analyzing a suitable mathematical model.
We are here concerned with the main span, namely the part of the roadway between the towers, which has a rectangular shape with two long edges
(of the order of 1km) and two shorter edges (of the order of 20m) fixed and hinged between the towers. Due to the large discrepancy between these measures we model
the roadway as a degenerate plate, that is,
a beam representing the midline of the roadway with cross sections which are free to rotate around the beam. We call this model a {\em fish-bone}, see Figure \ref{bonefish}.
\begin{figure}[ht]
\begin{center}
{\includegraphics[height=22mm, width=112mm]{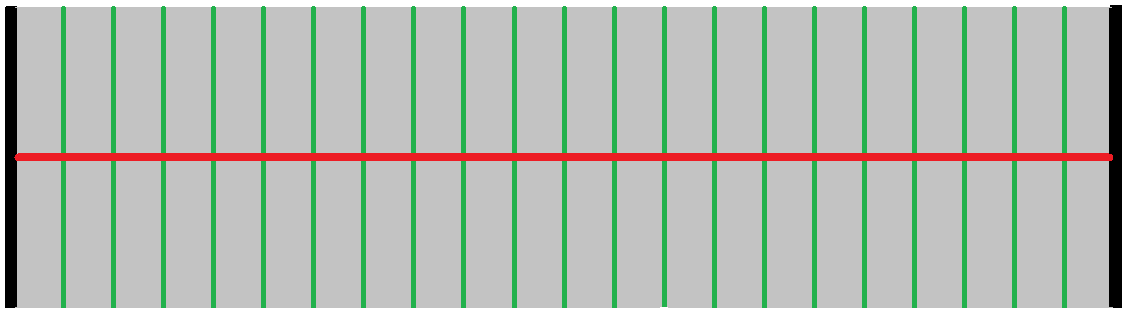}}
\caption{The model of a fish-bone plate.}\label{bonefish}
\end{center}
\end{figure}
The grey part is the roadway, the two black cross sections are between the towers, they are fixed and the plate is hinged there. The red line contains the barycenters
of the cross sections and is the line where the downwards vertical displacement $y$ is computed. The green orthogonal lines are virtual
cross sections seen as rods that can rotate around their barycenter, the angle of rotation with respect to the horizontal position being denoted by $\theta$.
We assume that the roadway has length $L$ and width $2\ell$ with $2\ell\ll L$.
The kinetic energy of a rotating object is $\frac12 J\dot{\theta}^2$, where $J$ is the moment of inertia and $\dot{\theta}$ is the angular velocity.
The moment of inertia of a rod of length $2\ell$ about the perpendicular axis through its center is given by $\frac{1}{3}M\ell^2$ where $M$ is the mass of the rod.
Hence, the kinetic energy of a rod having mass $M$ and half-length $\ell$, rotating about its center with angular velocity $\dot{\theta}$, is given by
$\frac{M}{6}\ell^2\dot{\theta}^2$. On the other hand, the bending energy of the beam depends on its curvature and this leads to a fourth order equation, see \cite{biot}.
Note that $M$ is also the mass per unit length in the longitudinal direction. The hangers are prestressed and the
equilibrium position of the midline is $y=0$, recall that $y>0$ corresponds to a downwards displacement of the midline. The equations for this system read
\begin{equation}\label{system0}
\left \{ \begin{array}{ll}
My_{tt}+EI y_{xxxx}+f(y+\ell\sin \theta)+f(y-\ell\sin \theta)=0\quad & 0<x<L\quad t>0\\
\frac{M\ell^2}{3}\theta_{tt}-\mu\ell^2\theta_{xx}+\ell\cos\theta\, (f(y+\ell\sin \theta)-f(y-\ell\sin \theta))=0\quad & 0<x<L\quad t>0,
\end{array}
\right.
\end{equation}
where $\mu>0$ is a constant depending on the shear modulus and the moment of inertia of the pure torsion, $EI>0$ is the flexural rigidity of the beam,
$f$ represents the restoring
action of the prestressed hangers and therefore also includes the action of gravity. We have not yet simplified by $\ell$ the second equation in \eq{system0} in order
to emphasize all the terms.\par
To \eq{system0} we associate the following boundary-initial conditions:
\begin{equation}\label{boundaryc}
y(0,t)=y_{xx}(0,t)=y(L,t)=y_{xx}(L,t)=\theta(0,t)=\theta(L,t)=0\qquad t\geq 0
\end{equation}
\begin{equation}\label{initial0}
y(x,0)=\eta_0(x)\,,\quad y_t(x,0)=\eta_1(x)\,,\quad\theta(x,0)=\theta_0(x)\,,\quad\theta_t(x,0)=\theta_1(x) \qquad 0<x<L\,.
\end{equation}
The first four boundary conditions in \eq{boundaryc} model a beam hinged at its endpoints whereas the last two boundary conditions model the fixed cross sections
between towers.\par
In a slightly different setting, involving mixed space-time fourth order derivatives, a {\em linear} version of \eq{system0} was first suggested by
Pittel-Yakubovich \cite{pittel}, see also \cite[Chapter VI]{yakubovich}; this model, with the addition of an external forcing representing the wind, was studied with
a parametric resonance approach and an instability was found for a sufficiently large action of the wind. This approach received severe criticisms from engineers
\cite[p.841]{scanlan}, see also \cite{billah,jenkins} for the physical point of view. The reason is that ``too much importance is attributed to the action of the wind'' as
if some kind of forced resonance would be involved. And it is clear that, in a windstorm, a precise phenomenon such as forced resonance is quite unlikely to be seen
\cite[Section 1]{mck1}. More recently, Moore \cite{moore} considered \eq{system0} with
$$f(s)=k\left[\left(s+\frac{Mg}{2k}\right)^+-\frac{Mg}{2k}\right],$$
a nonlinearity which models hangers behaving as linear springs of elastic constant $k>0$ if stretched but exert no restoring force if compressed; here $g$ is gravity. This
nonlinearity, first suggested by McKenna-Walter \cite{McKennaWalter}, describes the possible slackening of the hangers (occurring for $s\le-\frac{Mg}{2k}$) which was observed
during the Tacoma Bridge collapse, see \cite[V-12]{ammann}. But Moore considers the case where the hangers do not slacken: then $f$ becomes {\em linear}, $f(s)=ks$, and
the two equations in \eq{system0} decouple. In this situation there is obviously no interaction between vertical and torsional oscillations and, consequently,
no possibility to give an answer to {\bf (Q)}.\par
It is nowadays established that suspension bridges behave nonlinearly,
see \cite{brown,gazz,lacarbonara,McKennaWalter} and references therein. Whence, nonlinear restoring forces $f$ in \eq{system0} appear unavoidable if one wishes to have a
realistic model. A nonlinear $f$ was introduced in \eq{system0} by Holubov\'a-Matas \cite{holubova} who were able to prove well-posedness for a forced-damped version
of \eq{system0}.\par
For a slightly different model, numerical results obtained by McKenna \cite{mckmonth} show a sudden development of large torsional oscillations
as soon as the hangers lose tension, that is, as soon as the restoring force behaves nonlinearly. Further numerical results by Doole-Hogan \cite{doole}
and McKenna-Tuama \cite{mckO} show that a purely vertical periodic forcing may create a torsional response. An answer to {\bf (Q)} was recently given in
\cite{argaz} by using suitable Poincar\'e maps for a suspension bridge modeled by several coupled (second order) nonlinear oscillators. When enough energy
is present within the structure a resonance may occur, leading to an energy transfer between oscillators. The results in \cite{argaz} are, again, purely numerical.
So far no theoretical explanation of the origin of torsional oscillations has been given, nor any effective way to estimate the conditions which may create
torsional instability. This naturally leads to the following question (see \cite[Problem 7.4]{mck}): {\em can one employ the tools of nonlinear analysis to
say anything further in terms of stability}?\par
In this paper we consider the fish-bone model and we display the same phenomenon of sudden transition from purely vertical to torsional oscillations. Let us mention that a somehow related behavior of self-excited oscillations is visible in nonlinear beam equations,
see \cite{bfgk,pavani}. Here we provide a qualitative theoretical explanation of how internal resonances occur in \eq{system0}, yielding instability. Our results are purely qualitative and consider the bridge as an \emph{isolated system}, with no dissipation and no interaction with the surrounding fluid. We neglect the so-called aerodynamic forces and we focus our attention on the nonlinear structural behavior. In a forthcoming paper \cite{bergazpie} we will include aerodynamic forces and perform a more quantitative analysis, referring to actual suspension bridges.
\par
In Theorem \ref{Galerkin} we prove well-posedness of \eq{system0}-\eq{boundaryc}-\eq{initial0} for a wide class of nonlinearities $f$. The proof is based on a Galerkin
method which enables us to project \eq{system0} on a finite dimensional subspace of the phase space and to study the instability of the vertical modes
in terms of suitable Hill equations \cite{hill}.
After justifying (both physically and mathematically) this finite dimensional projection we show that it enables us to determine both
theoretical and numerical bounds for stability, see Sections \ref{finitedim}, \ref{onemode}, \ref{twomode}, and to explain the origin of torsional instability.\par
The obtained results yield the
following answer to question {\bf (Q)}. The onset of large torsional oscillations is due to a structural resonance which generates an energy transfer between different
oscillation modes. When the bridge
is oscillating vertically with sufficiently large amplitude, part of the energy is suddenly transferred to a torsional mode giving rise to wide torsional oscillations. And estimates of what is meant by ``large amplitudes'' may be obtained both theoretically and numerically.

\section{Simplification of the model and well-posedness}

It is not our purpose to give the precise quantitative behavior of the model under consideration. Therefore, in this section we make several
simplifications which do not modify the qualitative behavior of the nonlinear system \eq{system0}.\par
First of all, up to scaling we may assume that $L=\pi$; this will simplify the Fourier series expansion. Then we take
$EI\left(\frac{\pi}{L} \right)^{4}=3\mu\left(\frac{\pi}{L} \right)^{2}=1$ although these parameters may be fairly different in actual bridges.
Finally, note that the change of variable $t\mapsto \sqrt{M}t$ results in a positive or
negative delay in the occurrence of any (possibly catastrophic) phenomenon; whence, we may take $M=1$.\par
After these changes \eq{system0} becomes
\begin{equation}\label{system}
\left \{ \begin{array}{ll}
y_{tt}+ y_{xxxx}+f(y+\ell\sin \theta)+f(y-\ell\sin \theta)=0\quad & 0<x<\pi\quad t>0\\
\ell\theta_{tt}-\ell  \theta_{xx}+3\cos\theta\, (f(y+\ell\sin \theta)-f(y-\ell\sin \theta))=0\quad & 0<x<\pi\quad t>0,
\end{array}
\right.
\end{equation}
with boundary-initial conditions
\begin{equation}\label{boundary}
y(0,t)=y_{xx}(0,t)=y(\pi,t)=y_{xx}(\pi,t)=\theta(0,t)=\theta(\pi,t)=0\qquad t\geq 0
\end{equation}
\begin{equation}\label{initial}
y(x,0)=\eta_0(x)\,,\quad y_t(x,0)=\eta_1(x)\,,\quad\theta(x,0)=\theta_0(x)\,,\quad\theta_t(x,0)=\theta_1(x) \qquad 0<x<\pi\,.
\end{equation}

If $f$ is nondecreasing, as in the physical situation, then
\begin{equation}\label{F>0}
F(s):=\int_0^sf(\tau)\,d\tau\quad\mbox{is a positive convex function.}
\end{equation}
Therefore, the coercive functional (here $'=\frac{d}{dx}$)
$$J(y,\theta)=\frac{\|y''\|_2^2}{2}+\ell^2 \frac{\|\theta'\|_2^2}{6}+\int_0^{\pi}\left[F(y+\ell\sin \theta)
+F(y-\ell\sin \theta)\right]\,dx\, ,$$
defined for all $y\in H^2\cap H_0^1(0,\pi)$ and $\theta\in H_0^1(0,\pi)$,
admits a unique absolute minimum which coincides with the equilibrium $(y,\theta)=(0,0)$; here and in the sequel $\|\cdot\|_2$ denotes the
$L^2(0,\pi)$-norm.\par
We say that the functions
$$y\in C^0(\R_+;H^2\cap H_0^1(0,\pi))\cap  C^1(\R_+;L^2(0,\pi))\cap  C^2(\R_+;H^*(0,\pi))$$
$$\theta\in C^0(\R_+;H_0^1(0,\pi))\cap  C^1(\R_+;L^2(0,\pi))\cap  C^2(\R_+;H^{-1}(0,\pi))$$
are solutions of \eq{system}-\eq{boundary}-\eq{initial} if they satisfy the initial conditions \eq{initial} and if
$$\langle y_{tt}, \varphi\rangle_{H^*}+(y_{xx},\varphi'')+(f(y-\ell\sin \theta)+f(y+\ell\sin \theta),\varphi)=0\quad
\forall\,\varphi\in H^2\cap H_0^1(0,\pi)\,, \forall t>0\,,$$
$$\ell\langle \theta_{tt},\psi\rangle_{H^{-1}}+\ell(\theta_{x},\psi')+3 \cos \theta(f(y+\ell\sin \theta)-f(y-\ell\sin \theta),\psi)=0\quad \forall\,\psi\in H_0^1(0,\pi)\,, \forall t>0\,,$$
where $\langle\cdot,\cdot\rangle_{H^{-1}}$ and $\langle\cdot,\cdot\rangle_{H^*}$ are the duality pairings in $H^{-1}=(H_0^1(0,\pi))'$ and
$H^*=(H^2\cap H_0^1(0,\pi))'$ while $(\cdot,\cdot)$ denotes the scalar product in $L^2(0,\pi)$. We have

\begin{theorem}\label{Galerkin}
Let $\eta_0\in H^2 \cap H_0^1(0,\pi)$, $\theta_0\in H_0^1(0,\pi)$, $\eta_1,\theta_1\in L^2(0,\pi)$. Assume that $f\in{\rm Lip}_{loc}(\R)$ is nondecreasing, with
$f(0)=0$, and $|f(s)|\leq C(1+|s|^p)$ for every $s\in \R\setminus\{0\}$ and for some $p\geq1$. Then there exists a unique solution $(y,\theta)$ of \eqref{system}-\eqref{boundary}-\eqref{initial}.
\end{theorem}

The proof of Theorem \ref{Galerkin} is essentially due to \cite[Theorems 8 and 11]{holubova}. For the sake of completeness and since we require additional regularity for the solution, we quote a sketch of its proof in Section \ref{proofG}. It is based on a Galerkin procedure which suggests to approximate \eqref{system} with a finite dimensional system. In the next section, we study in some detail these approximate systems.

\section{Finite dimensional torsional stability}\label{finitedim}

\subsection{Dropping the trigonometric functions}\label{drop}

Since we are willing to describe how small torsional oscillations may suddenly become larger ones, we can use the following approximations:
\neweq{quasi}
\cos\theta \cong 1 \quad \text{and} \quad \sin \theta \cong \theta\,.
\endeq
This statement requires a rigorous justification.
It is known from the Report \cite[p.59]{ammann} that the torsional angle of the Tacoma Narrows Bridge prior to its collapse grew up until $45^\circ$.
On the other hand, Scanlan-Tomko \cite[p.1723]{scantom} judge that the torsional angle
can be considered harmless provided that it remains smaller than $3^\circ$. In radians this means that
\neweq{sentenza}
\mbox{the torsional angle may grow up until }\frac{\pi}{4}\mbox{ and may be considered harmless until }\frac{\pi}{60}\, .
\endeq

By the Taylor expansion with the Lagrange remainder term, we know that
\neweq{taylor}
\sin\eps=\sum_{k=0}^n(-1)^k\frac{\eps^{2k+1}}{(2k+1)!}+(-1)^{2n+3}\cos(\eps_\sigma)\frac{\eps^{2n+3}}{(2n+3)!}
:=P(\eps,n)+\Gamma_s(\eps,n)\qquad\forall\eps\in\R
\endeq
where $|\eps_\sigma|<|\eps|$ while $P$ and $\Gamma_s$ represent, respectively, the approximating polynomial and the approximating error. We have that
$$
P\left(\frac{\pi}{60},0\right)=\frac{\pi}{60}\, ,\quad
P\left(\frac{\pi}{60},1\right)=\frac{\pi}{60}-\frac{\pi^3}{1296}\cdot10^{-3}\, ,\quad
P\left(\frac{\pi}{4},0\right)=\frac{\pi}{4}\, ,\quad
P\left(\frac{\pi}{4},1\right)=\frac{\pi}{4}-\frac{\pi^3}{384}\, ,
$$
while we know that
$$\sin\frac{\pi}{60}=\sin\left(\frac{\pi}{10}-\frac{\pi}{12}\right)=
\frac{(\sqrt5 -1)(\sqrt6 +\sqrt2 )-(\sqrt6 -\sqrt2 )\sqrt{10+2\sqrt5}}{16}\approx0.0523
\, ,\quad\sin\frac{\pi}{4}=\frac{1}{\sqrt2}\, .$$
Therefore, the relative error $R_s(\eps,n):=|\frac{\sin\eps-P(\eps,n)}{\sin\eps}|$ (or percentage error) is given by
$$
R_s\left(\frac{\pi}{60},0\right)\approx4.6\cdot10^{-4}\, ,\quad R_s\left(\frac{\pi}{60},1\right)\approx6.3\cdot10^{-8}\, ,\quad
R_s\left(\frac{\pi}{4},0\right)\approx0.11\, ,\quad R_s\left(\frac{\pi}{4},1\right)\approx3.5\cdot10^{-3}\, .
$$

Similarly, we proceed with the cosine function. The Taylor expansion yields
$$
\cos\eps=\sum_{k=0}^n(-1)^k\frac{\eps^{2k}}{(2k)!}+(-1)^{2n+2}\sin(\eps_\sigma)\frac{\eps^{2n+2}}{(2n+2)!}
:=Q(\eps,n)+\Gamma_c(\eps,n)\qquad\forall\eps\in\R\, .
$$
We have that
$$
Q\left(\frac{\pi}{60},0\right)=1\, ,\quad
Q\left(\frac{\pi}{60},1\right)=1-\frac{\pi^2}{7200}\, ,\quad
Q\left(\frac{\pi}{4},0\right)=1\, ,\quad
Q\left(\frac{\pi}{4},1\right)=1-\frac{\pi^2}{32}\, ,
$$
while we also know that
$$\cos\frac{\pi}{60}\approx0.999\, ,\quad\cos\frac{\pi}{4}=\frac{1}{\sqrt2}\, .$$
Therefore, the relative error $R_c(\eps,n):=|\frac{\cos\eps-P(\eps,n)}{\cos\eps}|$ is given by
$$
R_c\left(\frac{\pi}{60},0\right)\approx1.4\cdot10^{-3}\, ,\quad R_c\left(\frac{\pi}{60},1\right)\approx3.1\cdot10^{-7}\, ,\quad
R_c\left(\frac{\pi}{4},0\right)\approx0.41\, ,\quad R_c\left(\frac{\pi}{4},1\right)\approx2.2\cdot10^{-2}\, .
$$

The above results enable us to draw the following conclusions, which we collect in a proposition.

\begin{proposition}\label{propapprox}\ \par\noindent
$\bullet$ If the model allows torsional angles up to $\frac{\pi}{4}$, then the approximation \eqref{quasi} is incorrect, yielding large relative
errors ($41\%$ for the cosine and $11\%$ for the sine); a second order approximation still yields fairly large relative errors ($2.2\%$ for the
cosine and $0.4\%$ for the sine).\par\noindent
$\bullet$ If the model allows torsional angles up to $\frac{\pi}{60}$, the approximation \eqref{quasi} is quite accurate, yielding small relative errors
($0.14\%$ for the cosine and less than $0.05\%$ for the sine); a second order approximation will not improve significantly the precision of the model.\par
\end{proposition}

Since the purpose of our numerical results is to consider small torsional data, of the order of $10^{-4}$, and since our purpose is merely to detect
when the torsional angle $\theta$ increases of two orders of magnitude, thereby reaching at most $10^{-2}\ll\frac{\pi}{60}$, we can make use of
the approximation \eq{quasi}. We emphasize that our results do not aim to describe the behavior of the bridge when the torsional angle becomes large,
they just aim to describe how a small torsional angle ceases to be small.

Proposition \ref{propapprox} allows us to implement the approximation suggested by \eq{quasi}; we set $z:=\ell \theta$ so that \eq{system} becomes
\begin{equation}\label{system A}
\left \{ \begin{array}{ll}
y_{tt}+y_{xxxx}+f(y+z)+f(y-z)=0\qquad & (0<x<\pi,\ t\ge0)\\
z_{tt}-z_{xx}+3f(y+z)-3f(y-z)=0\qquad & (0<x<\pi,\ t\ge0)\, .
\end{array}\right.
\end{equation}
In \eq{system A} the dependence on the width $\ell$ is somehow hidden; to recover this dependence, note that $\theta=\frac{z}{\ell}$ so that smaller $\ell$ yield larger $\theta$, that is, less stability.

\subsection{Choosing the nonlinearity}\label{gammano1}

We consider a specific nonlinearity $f$ satisfying the assumptions of Theorem \ref{Galerkin}. Since our purpose is merely to
describe the qualitative phenomenon, the choice of the nonlinearity is not of fundamental importance;
it is shown in \cite{argaz} that several different nonlinearities yield the same qualitative behavior for the solutions. We take
\neweq{fgamma}
f(s)=s+\gamma s^3\quad \text{for } \gamma>0\, ,
\endeq
which allows to simplify several computations. Let us also mention that Plaut-Davis \cite[Section 3.5]{plautdavis} make the same choice and that this
nonlinearity appears in several elastic contexts, see e.g.\ \cite[(1)]{ivanov}.\par
The parameter $\gamma$ measures how far is $f$ from a linear function. When $f$ is as in \eq{fgamma}, the system \eq{system A} becomes
\begin{equation}\label{systemgamma}
\left \{ \begin{array}{ll}
y_{tt}+y_{xxxx}+2y(1+\gamma y^2+3\gamma z^2)=0\qquad & (0<x<\pi,\ t\ge0)\\
z_{tt}-z_{xx}+6z(1+3\gamma y^2+\gamma z^2)=0\qquad & (0<x<\pi,\ t\ge0)\, .
\end{array}\right.
\end{equation}

To \eq{systemgamma} we associate some initial conditions which determine the conserved energy of the system, that is,
\begin{eqnarray*}
E_\gamma &=& \frac{\|y_t(t)\|_2^2}{2}+\frac{\|z_t(t)\|_2^2}{6}+\frac{\|y_{xx}(t)\|_2^2}{2}+\frac{\|z_x(t)\|_2^2}{6}\\
\ & \ & +\int_0^\pi \Big(y(x,t)^2+z(x,t)^2+3\gamma z(x,t)^2y(x,t)^2+\gamma\frac{y(x,t)^4}{2}+\gamma\frac{z(x,t)^4}{2} \Big)\, dx\, .
\end{eqnarray*}

Let $(y_\gamma,z_\gamma)$ be the solution of \eq{systemgamma} with some initial conditions. If we put
$(\overline{y},\overline{z})=\sqrt{\gamma}\,(y_{\gamma},y_{\gamma})$, then $(\overline{y},\overline{z})$ solves system \eq{systemgamma} when $\gamma=1$.
Accordingly, the conserved energy is modified:

\begin{proposition}\label{scaling}
Let $\gamma>0$. The conserved energy of \eqref{systemgamma} satisfies $E_\gamma=E_1/\gamma$, where $E_1$ is the conserved energy of \eqref{systemgamma}
when $\gamma=1$. Moreover, the widest vertical amplitude $\|y_\gamma\|_\infty$ satisfies $\|y_\gamma\|_\infty=\|\overline{y}\|_\infty/\sqrt{\gamma}$.
\end{proposition}

The proof of Proposition \ref{scaling} follows by rescaling. Proposition \ref{scaling} enables us to restrict our attention to the case $\gamma=1$, that is,
\begin{equation}\label{cubica}
f(s)=s+s^3\, .
\end{equation}
In this case we have
\begin{equation}\label{Gi2}
f(y+z)+f(y-z)=2y(1+y^2+3z^2)\quad \text{ and }\quad f(y+z)-f(y-z)=2z(1+3y^2+z^2)
\end{equation}
so that \eq{system} reduces to
\neweq{fcon1}
\left \{ \begin{array}{ll}
y_{tt}+y_{xxxx}+2y(1+y^2+3z^2)=0\qquad & (0<x<\pi,\ t\ge0)\\
z_{tt}-z_{xx}+6z(1+3y^2+z^2)=0\qquad & (0<x<\pi,\ t\ge0)\, .
\end{array}
\right.\end{equation}

Moreover, the conserved energy is given by
\begin{eqnarray}
E &=& \frac{\|y_t(t)\|_2^2}{2}+\frac{\|z_t(t)\|_2^2}{6}+\frac{\|y_{xx}(t)\|_2^2}{2}+\frac{\|z_x(t)\|_2^2}{6} \notag\\
\ & \ & +\int_0^\pi \left[\frac{y(x,t)^4}{2}+\frac{z(x,t)^4}{2}+3z(x,t)^2y(x,t)^2+y(x,t)^2+z(x,t)^2 \right]\, dx\, .\label{Entot}
\end{eqnarray}

Our purpose is to determine energy thresholds for torsional stability of vertical modes (still to be rigorously defined), see Section \ref{torstab}. From
Proposition \ref{scaling} we see that
$$\gamma\mapsto E_\gamma\qquad\mbox{and}\qquad\gamma\mapsto\|y_\gamma\|_\infty$$
are decreasing with respect to $\gamma$ and both tend to 0 if $\gamma\to\infty$, whereas they tend to $\infty$ if $\gamma\to0$. This shows that the
nonlinearity plays against stability:
\begin{center}
\begin{minipage}{165mm}
{\em more nonlinearity yields more instability and almost linear elastic behaviors are extremely stable.}
\end{minipage}
\end{center}

\subsection{Why can we neglect high torsional modes?}

Our finite dimensional analysis is performed on the low modes. This procedure is motivated by classical engineering literature.
Bleich-McCullough-Rosecrans-Vincent \cite[p.23]{bleich} write that {\em out of the infinite number of possible modes of motion in which a suspension bridge
might vibrate, we are interested only in a few, to wit: the ones having the smaller numbers of loops or half waves}. The physical reason why only low modes should
be considered is that higher modes require large bending energy; this is well explained by Smith-Vincent \cite[p.11]{tac2} who write that {\em the higher modes with
their shorter waves involve sharper curvature in the truss and, therefore, greater bending moment at a given amplitude and accordingly reflect the influence of the
truss stiffness to a greater degree than do the lower modes}. The suggestion to restrict attention to lower modes, mathematically corresponds to project an infinite
dimensional phase space on a finite dimensional subspace, a technique which should be attributed to Galerkin \cite{galer}.\par
Consider the solution $(y,z)$ of \eqref{fcon1}-\eqref{boundary}-\eqref{initial}, as given by a straightforward variant of Theorem \ref{Galerkin},
and let us expand it in Fourier series with respect to $x$:
\neweq{fourier}
y(x,t)=\sum_{j=1}^\infty y_j(t)\sin(jx)\ ,\quad z(x,t)=\sum_{j=1}^\infty z_j(t)\sin(jx)\ ,
\endeq
where the functions $y_j$ and $z_j$ are the unknowns. Denote by $z^m$ the projection of $z$ on the space spanned by $\{\sin(x),...,\sin(mx)\}$
and by $w^m$ the projection of $z$ on the infinite dimensional space spanned by $\{\sin((m+1)x),\sin((m+2)x)...\}$:
\neweq{split}
z(x,t)=z^m(x,t)+w^m(x,t)\, ,\quad z^m(x,t)=\sum_{j=1}^m z_j(t)\sin(jx)\, ,\quad w^m(x,t)=\sum_{j=m+1}^\infty z_j(t)\sin(jx)\, .
\endeq

In view of \eq{sentenza}, the next definition appears necessary: it characterizes solutions with small high torsional modes.

\begin{definition}\label{negligible}
Let $\omega>0$ and let $z\in C^0(\R_+;H_0^1(0,\pi))$. Let \eqref{split} be the decomposition $z$. We say that $z$ is
{\bf $\omega$-negligible above the $m$-th mode} if
$$\|w^m\|_\infty<\omega$$
where the $L^\infty$-norm is taken for $(x,t)\in(0,\pi)\times(0,+\infty)$.
\end{definition}

The choice of $\omega$ depends both on $\ell$ (through the substitution $z=\ell \theta$) and on the harmless criterion \eq{sentenza}, see Section \ref{drop}.
Nevertheless, since the purpose of the present paper is merely to give a qualitative description of the phenomena and of the corresponding procedures, we will
not quantify its value. In Section \ref{proofepsneg} we prove the following sufficient condition for a solution to be torsionally $\omega$-negligible on higher
modes.

\begin{theorem}\label{epsneg}
Let $\omega>0$ and let $(y,z)$ be a solution of \eqref{fcon1}-\eqref{boundary}-\eqref{initial} having energy $E>0$. Then
the torsional component $z$ is $\omega$-negligible above the $m$-th mode provided that at least one of the following inequalities holds
\neweq{prima}
\pi\, \omega^4\, (m+1)^2\, \Big[\pi(m^2+2m+7)^2+36E\Big]-36\, \pi^2\, E^2\, (m+1)^4-9\omega^8\ge0
\endeq
\neweq{seconda}
E^3+\frac{\pi}{2}\, E^2-\frac{3\, \omega^4}{4}\, E-\frac{3\, \omega^8}{32\, \pi}-\frac{\pi\, \omega^4}{3}\le0\, .
\endeq
\end{theorem}

The two inequalities \eq{prima} and \eq{seconda} have a completely different meaning. The condition \eq{seconda} is somehow obvious and uninteresting:
it states that if the total energy $E$ is sufficiently small then all the torsional components are small. In the next table we give some numerical bounds
for $E$ in dependence of the maximum allowed amplitude $\omega$.

\begin{center}
\begin{tabular}{|c|c|c|c|c|}
\hline
$\omega$ & $0.2$ & $0.1$ & $0.05$ & $0.01$ \\
\hline
$E$ & $3.3\cdot10^{-2}$ & $8.2\cdot10^{-3}$ & $2\cdot10^{-3}$ & $8.2\cdot10^{-5}$ \\
\hline
\end{tabular}
\par\vskip3mm
{\bf Upper bound for the energy $E$ in dependence of the maximal amplitude $\omega$.}
\end{center}

It appears clearly that the energy $E$ needs to be very small.\par
On the contrary, the condition \eq{prima} is much more useful: it gives an upper bound on the modes to be checked. High torsional modes remain small
provided they are above a threshold which depends on the energy $E$ and on the maximum allowed amplitude $\omega$. In the next two tables we give some
numerical bounds on the modes $m$ in dependence of the energy $E$, when $\omega$ is fixed.

\begin{center}
\begin{tabular}{|c|c|c|c|c|c|c|c|}
\hline
$E$ & $1$ & $0.5$ & $0.4$ & $0.3$ & $0.2$ & $0.1$ & $0.05$ \\
\hline
$m$ & $598$ & $298$ & $238$ & $178$ & $118$ & $58$ & $28$ \\
\hline
\end{tabular}
\par\vskip3mm
{\bf Upper bound for the number of modes $m$ in dependence of the energy $E$ when $\omega=0.1$.}
\end{center}
\par\vskip4mm
\begin{center}
\begin{tabular}{|c|c|c|c|c|c|c|c|}
\hline
$E$ & $1$ & $0.5$ & $0.4$ & $0.3$ & $0.2$ & $0.1$ & $0.05$ \\
\hline
$m$ & $148$ & $73$ & $58$ & $43$ & $28$ & $13$ & $5$ \\
\hline
\end{tabular}
\par\vskip3mm
{\bf Upper bound for the number of modes $m$ in dependence of the energy $E$ when $\omega=0.2$.}
\end{center}

It turns out that the map $E\mapsto m$ appears to be almost linear: in fact, we have
$$m\approx\frac{6E}{\omega^2}-2\, .$$
This approximation is reliable for small $\omega$: it follows by dropping the term $9\omega^8$ in \eq{prima}, by dividing by $(m+1)^2$,
and by solving the remaining second order algebraic inequality with respect to $m$.

\begin{remark} {\em With the very same procedure we may rule out high vertical modes where, possibly, \eq{seconda} becomes more useful.
We have here focused our attention only on the torsional modes because they are more dangerous for the safety of the bridge.}
\end{remark}

\subsection{Stability of the low modes}\label{torstab}

Let us fix some energy $E>0$. After having ruled out high modes (say, larger than $m$) through Theorem \ref{epsneg}, we focus our attention on the
lowest $m$ modes, $j\le m$. We consider the functions
\neweq{fouriern}
y^m(x,t)=\sum_{j=1}^m y_j(t)\sin(jx)\ ,\quad z^m(x,t)=\sum_{j=1}^m z_j(t)\sin(jx)
\endeq
aiming to approximate the solution of \eq{fcon1}, see the proof of Theorem \ref{Galerkin} in Section \ref{proofG}.
Put $(Y,Z):=(y_1,...,y_m,z_1,...,z_m)\in\R^{2m}$ and consider the system
\begin{equation}\label{systemU}
\left \{ \begin{array}{ll}
\ddot{y}_j(t)+j^4 y_j(t)+\frac{4}{\pi}\int_0^\pi y^m(x,t)(1\!+\!y^m(x,t)^2\!+\!3z^m(x,t)^2)\sin(jx)\, dx=0\\
\ddot{z}_j(t)+j^2 z_j(t)+\frac{12}{\pi}\int_0^\pi z^m(x,t)(1\!+\!3y^m(x,t)^2\!+\!z^m(x,t)^2)\sin(jx)\, dx=0
\end{array}\right.\quad(j=1,...,m).
\end{equation}
The proof of Theorem \ref{Galerkin} is constructive: (with minor changes) it ensures that $y^m$ and $z^m$ converge (as $m\to\infty$) to the unique solution of \eq{fcon1}.
To \eq{systemU} we associate the initial conditions
\neweq{icfin}
Y(0)=Y_0\, ,\quad \dot{Y}(0)=Y_1\, ,\quad Z(0)=Z_0\, ,\quad \dot{Z}(0)=Z_1\, ,
\endeq
where the components of the vector $Y_0\in \R^m$ are the Fourier coefficients of the projection of $y(0)$ onto the finite dimensional space spanned by
$\{\sin(jx)\}_{j=1}^m$; similarly for $Y_1$, $Z_0$, $Z_1$. In the sequel, we denote by
$$\{e_j\}_{j=1}^m\quad\mbox{the canonical basis of}\quad\R^m\, .$$
The conserved total energy of \eq{systemU}, to be compared with \eq{Entot}, is given by
\begin{eqnarray}
E &:=& \frac{|\dot{Y}|^2}{2}+\frac{|\dot{Z}|^2}{6}+\frac{1}{2}\sum_{j=1}^mj^4y_j^2+\frac{1}{6}\sum_{j=1}^mj^2z_j^2 \notag\\
\ &\ & +\frac{2}{\pi}\int_0^{\pi}\left[\frac{y^m(x,t)^4}{2}+\frac{z^m(x,t)^4}{2}+3y^m(x,t)^2z^m(x,t)^2+y^m(x,t)^2+z^m(x,t)^2\right]\, dx\, .\label{energy}
\end{eqnarray}

Note that \eq{energy} yields the boundedness of each of the $y_j$, $\dot{y}_j$, $z_j$, $\dot{z}_j$.
Once \eq{systemU} is solved, the functions $y^m$ and $z^m$ in \eq{fouriern} provide finite dimensional approximations of the solutions \eq{fourier}
of \eq{fcon1}. In view of Theorem \ref{epsneg}, this approximation is reliable since higher modes have small components.\par
Let us describe rigorously what we mean by vertical mode of \eq{systemU}.

\begin{definition}\label{oscillmode}
Let $m\ge1$ and $1\leq k\leq m$; let $\R^2\ni(\alpha,\beta)\neq(0,0)$. We say that $Y_k$ is the $k$-th \textbf{vertical mode} at energy $E_k(\alpha,\beta)$
if $(Y_k,0)\in\R^{2m}$ is the solution of \eqref{systemU} with initial conditions \eqref{icfin} satisfying
\neweq{infork}
Y(0)=\alpha e_k\, ,\quad \dot{Y}(0)=\beta e_k\, ,\quad Z(0)=\dot{Z}(0)=0\in\R^m\, .
\endeq
\end{definition}

By \eq{energy} and Lemma \ref{calculus} the conserved energy of \eq{systemU}-\eq{infork} is given by
\neweq{Eab}
E_k(\alpha,\beta):=\frac{\beta^2}{2}+(k^4+2)\frac{\alpha^2}{2}+\frac{3}{8}\alpha^4\, .
\endeq
The initial conditions in \eq{infork} determine the constant value of the energy $E_k(\alpha,\beta)$. Different couples of data $(\alpha,\beta)$ in \eq{infork}
may yield the same energy; in particular, for all $(\alpha,\beta)\in\R^2$ there exists a unique $\mu>0$ such that
\neweq{amplit}
E_k(\mu,0)=E_k(\alpha,\beta)\, .
\endeq
This value of $\mu$ is the amplitude of the initial oscillation of the $k$-the vertical mode.\par
The standard procedure to deduce the stability of $(Y_k,0)$ consists in studying the behavior of the perturbed vector $(Y-Y_k,Z)$ where $(Y,Z)$ solves \eqref{systemU},
see \cite[Chapter 5]{verhulst}. This leads to linearize the system \eqref{systemU} around $(Y_k,0)$ and, subsequently, to apply the Floquet theory for differential
equations with periodic coefficients, at least for $m=1,2$. The torsional components $\xi_j$ of the linearization of system \eqref{systemU} around $(Y_k,0)$ satisfy
\begin{equation}\label{inf syst2}
\frac{d^2 \Xi}{dt^2}+P_k(t)\,\Xi=0\,,
\end{equation}
where $\Xi=(\xi_1,..., \xi_m)$ and $P_k(t)$ is a $m\times m$ matrix depending on $Y_k$.

\begin{definition}\label{newdeff}
We say that the $k$-th vertical mode $Y_k$ at energy $E_k(\alpha,\beta)$ (that is, the solution of \eqref{systemU}-\eqref{infork})
is {\bf torsionally stable} if the trivial solution of \eqref{inf syst2} is stable.
\end{definition}

In the following two sections we state our (theoretical and numerical) stability results when $m=1$ and $m=2$. As we briefly explain in the Appendix, the cases
where $m\ge3$ are more involved because $Y$ may spread on more components; this will be discussed in a forthcoming paper \cite{bergazpie}. Our results lead to
the conclusion that
\begin{center}
\begin{minipage}{165mm}
{\em if the energy $E_k(\alpha,\beta)$ in \eqref{Eab} is small enough then small initial torsional oscillations remain small for all time $t>0$, whereas if
$E_k(\alpha,\beta)$ is large (that is, the vertical oscillations are initially large) then small torsional oscillations suddenly become wider.}
\end{minipage}
\end{center}

Therefore, a crucial role is played by the amount of energy inside the system \eq{fcon1}. In the next sections we analyze the energy, both theoretically and
numerically, within \eq{systemU} when $m=1$ and $m=2$. For the theoretical estimates of the critical energy we will make use of some stability criteria by
Zhukowski \cite{zhk} applied to suitable Hill equations \cite{hill}. For the numerical estimates, we choose ``small'' data $Z_0$ and $Z_1$ in \eq{icfin}
and, to evaluate the stability of the $k$-th vertical mode of \eq{systemU}, we consider data $Y_0$ and $Y_1$ concentrated on the $k$-th component of the
canonical basis of $\R^m$. More precisely, we take
\neweq{initialconcentrate}
Y_0=\mu e_k\, ,\qquad Y_1=0\in\R^m\, ,\qquad |Z_0|\le|\mu|\cdot10^{-4}\, ,\qquad |Z_1|\le|\mu|\cdot10^{-4}\, .
\endeq
Then the initial (and constant) energy \eq{energy} is approximately given by $E\approx(k^4+2)\frac{\mu^2}{2}+\frac38 \mu^4$ and the remaining (small)
part of the initial energy is the torsional energy of $Z_0$ and $Z_1$ plus some coupling energy. We also show that different
initial data, with $Y_1\neq 0$, give the same behavior provided the initial energy is the same.

\section{The 1-mode system}\label{onemode}

When $m=1$, the approximated 1-mode solutions \eq{fouriern} have the form
$$y^1(x,t)=y_1(t)\sin x\ ,\qquad z^1(x,t)=z_1(t)\sin x\ .$$
By Lemma \ref{calculus} in the Appendix, \eq{systemU} reads
\begin{equation}\label{unmodo}
\left\{\begin{array}{ll}
\ddot{y}_1+3y_1+\frac32 y_1^3+\frac92 y_1z_1^2=0\\
\ddot{z}_1+7z_1+\frac92 z_1^3+\frac{27}{2}z_1y_1^2=0\ ,
\end{array}\right.
\end{equation}
with some initial conditions
\neweq{initialone}
y_1(0)=\eta_0\, ,\ \dot{y}_1(0)=\eta_1\, ,\ z_1(0)=\zeta_0\, ,\ \dot{z}_1(0)=\zeta_1\ .
\endeq
Hence, in this case $Y_1=\overline{y}$ where $\overline{y}$ is the unique (periodic) solution of the autonomous equation
\neweq{soloy}
\ddot{y}+3y+\frac32 y^3=0\ ,\qquad y(0)=\alpha\, ,\ \dot{y}(0)=\beta\, ,
\endeq
which admits the conserved quantity
\neweq{energyy}
E=\frac{\dot{y}^2}{2}+\frac32 y^2+\frac38 y^4\equiv\frac{\beta^2}{2}+\frac32 \alpha^2+\frac38 \alpha^4\, .
\endeq
Therefore, \eq{inf syst2} reduces to the following Hill equation \cite{hill}:
\neweq{nonlinearhill}
\ddot{\xi}+a(t)\xi=0\quad\mbox{with}\quad a(t)=7+\frac{27}{2}\overline{y}(t)^2\,,
\endeq

In Section \ref{5} we prove

\begin{theorem}\label{stable}
The first vertical mode $Y_1=\overline{y}$ at energy $E_1(\alpha,\beta)$ (that is, the solution of \eqref{soloy})
is torsionally stable provided that
$$\|\overline{y}\|_\infty\le\sqrt{\frac{10}{21}}\approx0.69$$
or, equivalently, provided that
$$E_1\le\frac{235}{294}\approx0.799\, .$$
\end{theorem}

As already remarked, Definition \ref{newdeff} is the usual one. Nevertheless, the stability results obtained in \cite{ortega} for suitable
\emph{nonlinear} Hill equations suggest that different equivalent definitions can be stated, possibly not involving a linearization process.
In particular, by \cite{ortega} we know that the stability of the trivial solution of \eq{nonlinearhill} implies the stability of the trivial solution of
$$
\ddot{\xi}+a(t)\xi+\frac{9}{2}\xi^3=0\quad\mbox{with}\quad a(t)=7+\frac{27}{2}\overline{y}(t)^2\, .
$$
We also refer to \cite{chu} and references therein for stability results for nonlinear first order
planar systems. As far as we are aware, there is no general
theory for nonlinear systems of any number of equations but it is reasonable to expect that similar results might hold.
This is why, in our numerical experiments, we consider system \eq{unmodo} without any linearization. The below numerical results
suggest that the threshold of instability is larger than the one in Theorem \ref{stable}. Clearly, they only give a ``local stability'' information (for finite time),
but the observed phenomenon is very precise and the thresholds of torsional instability are determined with high accuracy. The pictures in Figure \ref{tantefig}
display the plots of the solutions of \eq{unmodo} with initial data
\neweq{numinitial}
y_1(0)=\|y_1\|_\infty=10^4z_1(0)\, ,\quad \dot{y}_1(0)=\dot{z}_1(0)=0
\endeq
for different values of $\|y_1\|_\infty$. The green plot is $y_1$ and the black plot is $z_1$.
\begin{figure}[ht]
\begin{center}
{\includegraphics[height=23mm, width=41mm]{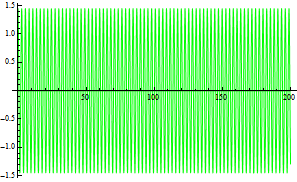}} {\includegraphics[height=23mm, width=41mm]{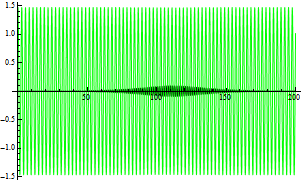}}
{\includegraphics[height=23mm, width=41mm]{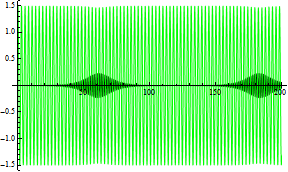}} {\includegraphics[height=23mm, width=41mm]{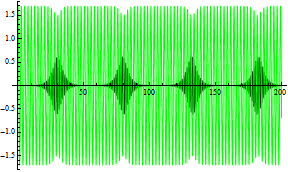}}
\caption{On the interval $t\in[0,200]$, plot of the solutions $y_1$ (green) and $z_1$ (black) of \eq{unmodo}-\eq{numinitial} for
$\|y_1\|_\infty=1.45,\, 1.47,\, 1.5,\, 1.7$ (from left to right).}\label{tantefig}
\end{center}
\end{figure}
For $\|y_1\|_\infty=1.45$ no wide torsion appears, which means that the solution $(y_1,0)$ is torsionally stable. For $\|y_1\|_\infty=1.47$
we see a sudden
increase of the torsional oscillation around $t\approx50$. Therefore, the stability threshold for the vertical amplitude of oscillation lies in the
interval $[1.45,1.47]$. Finer experiments show that the threshold is $\|y_1\|_\infty\approx1.46$, corresponding to a critical energy of about $E\approx4.9$:
these values should be compared
with the statement of Theorem \ref{stable}. When the amplitude is increased further, for $\|y_1\|_\infty=1.5$ and $\|y_1\|_\infty=1.7$, the appearance of
wide torsional oscillations is anticipated (earlier in time) and amplified (larger in magnitude). This phenomenon continues to increase for increasing
$\|y_1\|_\infty$. We then tried different initial
data with $\dot{y}_1(0)\neq0$; as expected, the sudden appearance of torsional oscillations always occurs at the energy level $E\approx4.9$, no matter of how
it is initially distributed between kinetic and potential energy of $y_1$. Summarizing, we have seen that the ``true'' (numerical) thresholds are larger than
the ones obtained in Theorem \ref{stable}. \par\smallskip
Let us give a different point of view of this phenomenon. If we slightly modify the parameters involved we can prove that the nonlinear frequency
of $z_1$ is larger than the frequency of $y_1$, which shows that the mutual position of $z_1$ and $y_1$ varies and may create the spark for an energy transfer.
Instead of $EI=3\mu=1$, we take $EI=\mu=1$ and \eqref{unmodo} should then be replaced by
\begin{equation}\label{unmodobis}
\left\{\begin{array}{ll}
\ddot{y}_1+3y_1+\frac32 y_1^3+\frac92 y_1z_1^2=0\\
\ddot{z}_1+9z_1+\frac92 z_1^3+\frac{27}{2}z_1y_1^2=0\ .
\end{array}\right.
\end{equation}

Then we prove

\begin{proposition}\label{signchange}
Let $(y_1,z_1)$ be a nontrivial solution of \eqref{unmodobis}. Let $t_1<t_2$ be
two consecutive critical points of $y_1(t)$. Then there exists $\tau\in(t_1,t_2)$ such that $z_1(\tau)=0$.
\end{proposition}
\proof For any solution $(y_1,z_1)$ of \eqref{unmodobis} we have
\begin{equation}\label{plapd}
\frac{d}{dt}\Big[\dot{y}_1^3\dot{z}_1\Big]+9\frac{d}{dt}\left[y_1z_1+\frac{y_1z_1^3}{2}+\frac{z_1y_1^3}{2}\right]\, \dot{y}_1^2=0\, .
\end{equation}
By integrating \eqref{plapd} by parts over $(t_1,t_2)$ we obtain
\begin{eqnarray*}
0 &=& 9\int_{t_1}^{t_2}\frac{d}{dt}\left[y_1(t)z_1(t)+\frac{y_1(t)z_1(t)^3}{2}+\frac{z_1(t)y_1(t)^3}{2}\right]\, \dot{y}_1(t)^2\, dt\\
\ &=& -18\int_{t_1}^{t_2}y_1(t)z_1(t)\left[1+\frac{z_1(t)^2}{2}+\frac{y_1(t)^2}{2}\right]\, \ddot{y}_1(t)\dot{y}_1(t)\, dt\\
\ &=& 54\int_{t_1}^{t_2}y_1(t)^2z_1(t)\left[1+\frac{z_1(t)^2}{2}+\frac{y_1(t)^2}{2}\right]\left[1+\frac{y_1(t)^2}{2}+\frac{3z_1(t)^2}{2}\right]\, \dot{y}_1(t)\, dt
\end{eqnarray*}
where, in the last step, we used \eqref{unmodobis}$_1$. In the integrand, $y_1^2[1+\frac{z_1^2}{2}+\frac{y_1^2}{2}][1+\frac{y_1^2}{2}+\frac{3z_1^2}{2}]\ge0$ and also
$\dot{y}_1$ has fixed sign so that the integral may vanish only if $z_1(t)$ changes sign in $(t_1,t_2)$.\endproof

Proposition \ref{signchange} shows that the nonlinear frequency of $z_1$ is always larger than the frequency of $y_1$. If the frequency of $z_1$ reaches
a multiple of the frequency of $y_1$ then an internal resonance is created and this yields a possible transfer of energy from $y_1$ to $z_1$.

\section{The 2-modes system}\label{twomode}

Let us fix $m=2$ in \eq{fouriern} and \eq{systemU} and put
$$y^2(x,t)=y_1(t)\sin x+y_2(t)\sin(2x)\, ,\qquad z^2(x,t)=z_1(t)\sin x+z_2(t)\sin(2x)\,.$$
Then, after integration over $(0,\pi)$ and using Lemma \ref{calculus}, we see that $y_j$ and $z_j$ satisfy the system\renewcommand{\arraystretch}{2.5}
\begin{equation}\label{coeff system 2}
\left \{ \begin{array}{ll}
\displaystyle{\ddot{y}_1+3y_1+\frac{9}{2}y_1z_1^2+3y_1z_2^2+3y_1y_2^2+\frac{3}{2}y_1^3+6z_1z_2y_2=0}\\
\displaystyle{\ddot{y}_2+18y_2+\frac{9}{2}y_2z_2^2+3y_2z_1^2+3y_2y_1^2+\frac{3}{2}y_2^3+6z_1z_2y_1=0}\\
\displaystyle{\ddot{z}_1+7z_1+\frac{27}{2}z_1y_1^2+9z_1y_2^2+9z_1z_2^2+\frac{9}{2}z_1^3+18y_1y_2z_2=0} \\
\displaystyle{\ddot{z}_2+10z_2+\frac{27}{2}z_2y_2^2+9z_2y_1^2+9z_2z_1^2+\frac{9}{2}z_2^3+18y_1y_2z_1=0}\,,
\end{array}\right.
\end{equation}\renewcommand{\arraystretch}{1.5}
while the energy becomes
\begin{eqnarray*}
E &=& \frac{1}{2}(\dot{y}_1^2+\dot{y}_2^2)+\frac{1}{6}(\dot{z}_1^2+\dot{z}_2^2)+\frac{3}{2}y_1^2+9y_2^2
+\frac{7}{6}z_1^2+\frac{5}{3}z_2^2+6y_1y_2z_1z_2\\
&\ &+\frac{9}{4}(y_1^2z_1^2+y_2^2z_2^2)+\frac{3}{2}(y_1^2y_2^2+y_1^2z_2^2+y_2^2z_1^2+z_1^2z_2^2)+\frac{3}{8}(y_1^4+y_2^4+z_1^4+z_2^4)\,.
\end{eqnarray*}
Since our purpose is to emphasize perturbations of linear equations, it is more convenient to rewrite the two last equations in \eq{coeff system 2} as
\renewcommand{\arraystretch}{2.2}
\begin{equation}\label{linear2modes}
\left \{ \begin{array}{ll}
\displaystyle{\ddot{z}_1+\left(7+\frac{27}{2}y_1^2+9y_2^2+9z_2^2\right)z_1+\frac{9}{2}z_1^3=-18y_1y_2z_2} \\
\displaystyle{\ddot{z}_2+\left(10+\frac{27}{2}y_2^2+9y_1^2+9z_1^2\right)z_2+\frac{9}{2}z_2^3=-18y_1y_2z_1}\,.
\end{array}\right.
\end{equation}
Hence, in this case we have $Y_j=\overline{y}_je_j$ ($j=1,2$) where $\overline{y}_j$ is the unique (periodic) solution of the problem
$$
\ddot{y}_j(t)+(j^4+2)y_j(t)+\frac{3}{2}y_j(t)^3=0\, ,\quad y_j(0)=\alpha\, ,\quad\dot{y}_j(0)=\beta\, ,
$$
which admits the conserved energy
\begin{equation}\label{energyEj}
E_j=\frac{\dot{y}(t)^2}{2}+(j^4+2)\frac{y(t)^2}{2}+\frac{3 y(t)^4}{8}=\frac{\beta^2}{2}+(j^4+2)\frac{\alpha^2}{2}+\frac{3}{8}\alpha^4\geq 0\, .
\end{equation}
Then \eq{inf syst2} becomes the following system of uncoupled Hill equations:
\neweq{nonlinearhillsystem2}\renewcommand{\arraystretch}{1.5}
\left \{ \begin{array}{ll}
\ddot{\xi}_1(t)+a_{1,j}(t) \xi_1(t)=0\\
\ddot{\xi}_2(t)+a_{2,j}(t) \xi_2(t)=0
\end{array}\right.\qquad(j=1,2)
\endeq
where $a_{i,j}(t)=i^2+6+9\alpha_{i,j}\overline{y}_j(t)^2$, $\alpha_{i,j}=1$ if $i\neq j$ and $\alpha_{i,i}=\frac{3}{2}$.\par
In Section \ref{22} we prove

\begin{theorem}\label{stable2}
The first vertical mode $Y_1\!\!=\!\!(\overline{y}_1,0)$ of \eqref{coeff system 2} at energy $E_1$ is torsionally stable provided that
\neweq{bounds2m}
\|\overline{y}_1\|_\infty\le\frac{1}{\sqrt 3}\approx0.577\ \Longleftrightarrow\ E_1\le\frac{13}{24}\approx0.542\, .
\endeq
The second vertical mode $Y_2\!\!=\!\!(0,\overline{y}_2)$ of \eqref{coeff system 2} at energy $E_2$ is torsionally stable provided that
$$\|\overline{y}_2\|_\infty\le\sqrt{\frac{32}{51}}\approx0.792\ \Longleftrightarrow\ E_2\le\frac{5024}{867}\approx5.795\, .$$
\end{theorem}

Again, Theorem \ref{stable2} merely gives a sufficient condition for the torsional stability and, numerically, the thresholds seem to be larger.
Once more, numerics only shows local stability but the observed phenomena are very precise and hence they appear reliable.
Here the situation is slightly more complicated because two modes (4 equations) are involved. Therefore, we proceed differently. \par
We start by studying the stability of the first vertical mode. The pictures in Figure \ref{tantefig2} display the plots
of the torsional components $(z_1,z_2)$ of the solutions of \eq{coeff system 2} with initial data
\neweq{numinitial2}
y_1(0)=\|y_1\|_\infty=10^4y_2(0)=10^4z_1(0)=10^4z_2(0)\, ,\quad \dot{y}_1(0)=\dot{y}_2(0)=\dot{z}_1(0)=\dot{z}_2(0)=0
\endeq
for different values of $\|y_1\|_\infty$. The green plot is $z_1$ and the black plot is $z_2$.
\begin{figure}[ht]
\begin{center}
{\includegraphics[height=23mm, width=41mm]{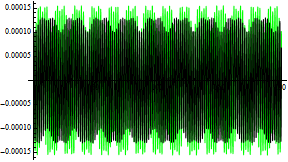}} {\includegraphics[height=23mm, width=41mm]{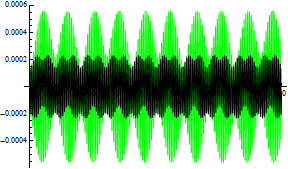}}
{\includegraphics[height=23mm, width=41mm]{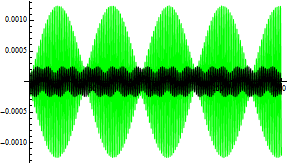}} {\includegraphics[height=23mm, width=41mm]{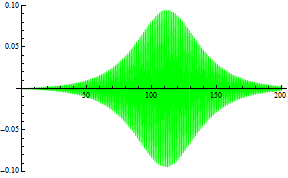}}
\caption{On the interval $t\in[0,200]$, plot of the torsional components $z_1$ (green) and $z_2$ (black) of \eq{coeff system 2}-\eq{numinitial2} for
$\|y_1\|_\infty=1,\, 1.4,\, 1.45,\, 1.47$ (from left to right).}\label{tantefig2}
\end{center}
\end{figure}
Recalling that the initial torsional amplitudes are of the order of $10^{-4}$ we can see that, for $\|y_1\|_\infty=1$, both torsional components remain small,
although $z_1$ is slightly larger than $z_2$. By increasing the $y_1$ amplitude, $\|y_1\|_\infty=1.4$ and $\|y_1\|_\infty=1.45$, we see that $z_1$ and $z_2$ still remain
small but now $z_1$ is significantly larger than $z_2$ and displays bumps. When $\|y_1\|_\infty=1.47$, $z_1$ has become so large that $z_2$, which is still of
the order of $10^{-4}$, is no longer visible in the fourth plot of Figure \ref{tantefig2}. The threshold for the appearance of $z_1\gg z_2$ is again
$\|y_1\|_\infty\approx1.46$, see Section \ref{onemode}. Therefore, it seems that the stability of the first vertical mode does not transfer energy
on the second modes; but, as we now show, this is not true.\par
We increased further the initial datum up to $\|y_1\|_\infty=3$. In Figure \ref{duefig} we display the plot of all the components
$(y_1,y_2,z_1,z_2)$ of the corresponding solution of \eq{coeff system 2}-\eq{numinitial2}.
\begin{figure}[ht]
\begin{center}
{\includegraphics[height=23mm, width=41mm]{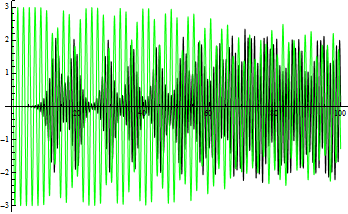}}\qquad {\includegraphics[height=23mm, width=41mm]{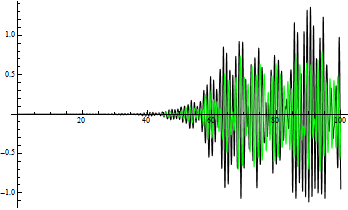}}
\caption{On the interval $t\in[0,100]$, plot of the solution of \eq{coeff system 2}-\eq{numinitial2} for $\|y_1\|_\infty=3$. Left picture:
green=$y_1$, black=$z_1$. Right picture: green=$y_2$, black=$z_2$.}\label{duefig}
\end{center}
\end{figure}
One can see that some energy is also transferred to both the vertical and torsional second modes, although this occurs with some delay (in the second picture,
the green oscillation is hidden but it is almost as wide as the black oscillation).\par
Concerning the stability of the second vertical mode, we just quickly describe our numerical results. The loss of stability appeared for
$\|y_2\|_\infty\approx0.945$ corresponding to $E\approx8.33$; in this case, Theorem \ref{stable2} gives a fairly good sufficient condition. For
$\|y_2\|_\infty\le0.94$, both $z_1$ and $z_2$ (and also $y_1$) remain small and of the same magnitude, with the amplitude of oscillations of $z_1$ being almost
constant while the amplitude of oscillations of $z_2$ being variable. For $\|y_2\|_\infty\ge0.945$, $z_2$ suddenly displays the bumps seen in the above pictures.
Finally, for $\|y_2\|_\infty\ge1.08$, also $y_1$ and $z_1$ display sudden wide oscillations which, however, appear delayed in time when compared with $z_2$.

\section{Proof of Theorem \ref{Galerkin}}\label{proofG}

The existence and uniqueness issues are inspired to \cite[Theorems 8 and 11]{holubova} while the regularity statement is achieved by arguing as in
\cite[Lemma 8.1]{FerGaz}.\par For the existence part we perform a Galerkin procedure. The sequence $\{\sin(jx)\}_{j\geq1}$ is an orthogonal basis
of the spaces $L^2(0,\pi), H_0^1(0,\pi)$ and $H^2\cap H_0^1(0,\pi)$. Then, for a given $n\in \N$, we set
\begin{equation}
\label{ym}
y^n(x,t)=\sum_{j=1}^n y_j(t)\sin(jx)\ ,\quad \theta^n(x,t)=\sum_{j=1}^n \theta_j(t)\sin(jx)\,,
\end{equation}
where $y_j$ and $\theta_j$ satisfy the system of ODE's
\begin{equation}
\label{coeff system}
\left \{ \begin{array}{ll}
\ddot{y}_j(t)+j^4 y_j(t)+\frac{2}{\pi}\int_0^{\pi}[f(y^n(x,t)+\ell\sin \theta^n(x,t))+f(y^n(x,t)-\ell\sin \theta^n(x,t))]\, \sin(jx)\,dx=0 \\
\ell\ddot{\theta}_j(t)+j^2\ell \theta_j(t)+\frac{6}{\pi}\int_0^{\pi}[f(y^n(x,t)+\ell\sin \theta^n(x,t))-f(y^n(x,t)-\ell\sin \theta^n(x,t))]\, \sin(jx)\,dx=0
\end{array}\right.
\end{equation}
for $t>0$ and $j=1,...,n$.
Moreover, writing the Fourier expansion of the initial data \eq{initial} as
$$\eta_0(x)=\sum_{j=1}^\infty \eta_0^j\sin(jx)\quad\text{in }H^2\cap H_0^1(0,\pi)\qquad y_1(x)=\sum_{j=1}^\infty \eta_1^j\sin(jx)\quad\text{in }L^2(0,\pi)$$
$$\theta_0(x)=\sum_{j=1}^\infty \theta_0^j \sin(jx) \quad \text{in }H_0^1(0,\pi)\qquad \theta_1(x)=\sum_{j=1}^\infty \theta_1^j \sin(jx)\quad\text{in }L^2(0,\pi)\,,$$
we assume that, for every $1\leq j\leq n$, the $y_j$'s and the $\theta_j$'s satisfy
\begin{equation}\label{initial condition m}
y_j(0)= \eta_0^j\,, \quad \dot{y}_j(0)= \eta_1^j\,,\quad \theta_j(0)= \theta_0^j\,, \quad \dot{\theta}_j(0)= \theta_1^j\,.
\end{equation}

The existence of a unique local solution of \eq{coeff system}-\eq{initial condition m} in some maximal interval of continuation
$[0, \tau_n)$, $\tau_n>0$, follows from standard theory of ODE's. Then, we multiply the first equation in \eq{coeff system} by $6\dot{y}_j(t)$ and
the second equation by $2\dot{\theta}_j(t)$, then we add the so obtained $2n$ equations for $j=1$ to $n$, finally we integrate over $(0,t)$ to obtain
$$ \begin{array}{llll}
3\|\dot{y}^n\!(t)\|_2^2\!+\!3\|y_{xx}^n\!(t)\|_2^2\!+\!\ell^2\|\dot{\theta}^n\!(t)\|_2^2\!+\!\ell^2\|\theta_x^n\!(t)\|_2^2\!+\!
12\int_0^{\pi}\!\left[F(y^n\!(t)\!+\!\ell\sin \theta^n\!(t))\!+\!F(y^n\!(t)\!-\!\ell\sin \theta^n\!(t)) \right]\!dx\!\leq\\
3\|\dot{y}^n\!(0)\|_2^2\!+\!3\|y_{xx}^n\!(0)\|_2^2\!+\!\ell^2\|\dot{\theta}^n\!(0)\|_2^2\!+\!\ell^2\|\theta_x^n\!(0)\|_2^2\!+\!
12\int_0^{\pi}\!\left[ F(y^n\!(0)\!+\!\ell\sin \theta^n\!(0))\!+\!F(y^n\!(0)\!-\!\ell\sin \theta^n\!(0)) \right]\!dx
\end{array}$$
where $y^n(t)=y^n(x,t)$, $\theta^n(t)=\theta^n(x,t)$ and $F(s)=\int_0^s f(\tau)\,d\tau$. Since $F\geq 0$, this yields
\begin{equation}\label{uniform bound}
3\|\dot{y}^n(t)\|_2^2+3\|y_{xx}^n(t)\|_2^2+\ell^2\|\dot{\theta}^n(t)\|_2^2+\ell^2\|\theta_x^n(t)\|_2^2\leq C \quad \text{for any } t\in[0,\tau_n) \text{ and } n\geq 1\,
\end{equation}
for some constant $C$ independent of $n$ and $t$. Hence, $\{y^n\}$ and $\{\theta^n \}$ are globally defined in $\R_+$ and uniformly bounded,
respectively, in the spaces $C^0([0,T];H^2\cap H_0^1(0,\pi))\cap C^1([0,T];L^2(0,\pi))$ and $C^0([0,T];H_0^1(0,\pi))\cap C^1([0,T];$ $L^2(0,\pi))$ for all
finite $T>0$. We show that they both admit a strongly convergent subsequence in the same spaces.\par
The estimate \eq{uniform bound} shows that $\{y^n\}$ and $\{\theta^n \}$ are bounded and equicontinuous in $C^0([0,T];L^2(0,\pi))$. By the Ascoli-Arzel\`a Theorem
we then conclude that, up to a subsequence, $y^n\rightarrow y$ and $\theta^n\rightarrow \theta$ strongly in $C^0([0,T];L^2(0,\pi))$.
By \eq{uniform bound} and the embedding $H^1_0(0,\pi)\subset L^\infty(0,\pi)$ we also infer that $y^n$ and $\theta^n$ are uniformly bounded in
$[0,\pi]\times[0,T]$. Whence,
$$
\left|\int_0^{\pi} F(y^n(t)+\ell\sin \theta^n(t))\,dx -\int_0^{\pi} F(y(t)+\ell\sin \theta(t))\,dx\right|
$$
$$\le\int_0^{\pi} |f(\tau(y^n(t)+\ell\sin \theta^n(t))+(1-\tau)(y(t)+\ell\sin \theta(t)))|\left(|y^n(t)-y(t)|+\ell |\theta^n(t)- \theta(t)|\right)\, dx$$
for some $\tau=\tau(x,t,n)\in[0,1]$. Since $y^n$ and $\theta^n$ are uniformly bounded, so is $f(\tau(y^n+\ell\sin \theta^n)+(1-\tau)(y+\ell\sin \theta))$
and the latter inequality yields
\neweq{firstbound}
\left|\int_0^{\pi} F(y^n(t)+\ell\sin \theta^n(t))\,dx -\int_0^{\pi} F(y(t)+\ell\sin \theta(t))\,dx\right|\le C\Big(\|y^n(t)-y(t)\|_2+ \ell\|\theta^n(t)-\theta(t)\|_2\Big)\to0\,.
\endeq
We may argue similarly for $F(y^n-\ell \sin \theta^n)$.\par
Next, for every $n>m\ge1$, we set $y^{n,m}:=y^n-y^m$ and $\theta^{n,m}:=\theta^n-\theta^m$. Repeating the computations which yield \eq{uniform bound}, for all $t\in[0,T]$ one gets
$$\begin{array}{llll}
3\|\dot{y}^{n,m}(t)\|_2^2+3\|y_{xx}^{n,m}(t)\|_2^2+\ell^2\|\dot{\theta}^{n,m}(t)\|_2^2+\ell^2\|\theta_x^{n,m}(t)\|_2^2=
\\3\|\dot{y}^{n,m}(0)\|_2^2+3\|y_{xx}^{n,m}(0)\|_2^2+\ell^2\|\dot{\theta}^{n,m}(0)\|_2^2+\ell^2\|\theta_x^{n,m}(0)\|_2^2-
\\12\int_0^{\pi} \left[F(y^n(t)\!+\!\ell\sin \theta^n(t))\!-\!F(y^m(t)\!+\!\ell\sin \theta^m(t))\!+\!F(y^n(t)\!-\!\ell\sin \theta^n(t))\!-\!
F(y^m(t)\!-\!\ell\sin \theta^m(t))\right]\!+
\\12\int_0^{\pi} \left[F(y^n(0)\!+\!\ell\sin \theta^n(0))\!-\!F(y^m(0)\!+\!\ell\sin \theta^m(0))\!+\!F(y^n(0)\!-\!\ell\sin \theta^n(0))\!-\!
F(y^m(0)\!-\!\ell\sin \theta^m(0))\right].
\end{array}$$
Therefore, by using \eq{firstbound}, we infer that
$$
\sup_{t\in[0,T]}\ \Big(3\|\dot{y}^{n,m}(t)\|_2^2+3\|y_{xx}^{n,m}(t)\|_2^2+\ell^2\|\dot{\theta}^{n,m}(t)\|_2^2+\ell^2\|\theta_x^{n,m}(t)\|_2^2\Big)\to0\qquad\mbox{as }n,m\to\infty
$$
so that $\{y^n\}$ and $\{\theta^n \}$ are Cauchy sequences in the spaces
$C^0([0,T];H^2\cap H_0^1(0,\pi))\cap C^1([0,T];L^2(0,\pi))$ and $C^0([0,T];H_0^1(0,\pi))\cap C^1([0,T];L^2(0,\pi))$, respectively. In turn this yields
$$y^n\rightarrow y \text{ in } C^0([0,T];H^2\cap H_0^1(0,\pi))\cap C^1([0,T];L^2(0,\pi))\quad \text{as }n\rightarrow +\infty\, ,$$
$$\theta^n\rightarrow \theta \text{ in } C^0([0,T];H_0^1(0,\pi))\cap C^1([0,T];L^2(0,\pi))\quad \text{as }n\rightarrow +\infty\,.$$
Let $\Upsilon\in C_c^{\infty}(0,T)$, $\varphi\in H^2\cap H_0^1(0,\pi)$ and $\psi\in H_0^1(0,\pi)$. We denote by $\varphi^n$ and $\psi^n$ the
orthogonal projections of $\varphi$ and
$\psi$ onto $X_n:={\rm span}\{\sin(jx)\}_{j=1}^n$ from, respectively, the spaces $H^2\cap H_0^1(0,\pi)$ and $H_0^1(0,\pi)$.
Then \eq{coeff system} yields
$$\left \{ \begin{array}{llll}
\int_0^T(\dot{y}^n(t),\varphi^n)\dot{\Upsilon}(t) \,dt=\\
\int_0^T \left[(y_{xx}^n(t),(\varphi^n)'')+\frac{2}{\pi}
(f(y^n(t)-\ell\sin \theta^n(t))+f(y^n(t)+\ell\sin \theta^n(t)),\varphi^n)\right]\Upsilon(t)\,dt\\
\int_0^T(\dot{\theta}^n(t),\psi^n)\dot{\Upsilon}(t)\,dt=\\
-\int_0^T\left[(\theta_x^n(t),(\psi^n)')+\frac{6}{\pi\ell}
(f(y^n(t)-\ell\sin \theta^n(t))-f(y^n(t)+\ell\sin \theta^n(t)),\psi^n)\right]\Upsilon(t)\,dt .
\end{array}\right.$$

Since, by compactness,
$$f(y^n\pm \ell\sin \theta^n)\rightarrow f(y\pm\ell\sin \theta)\quad \text{in }C^0([0,T],L^2(0,\pi))\,,$$
by letting $n\rightarrow +\infty$ in the above system we conclude that $y_{tt}\in C^0([0,T];H^*)$ and $\theta_{tt}\in C^0([0,T];H^{-1})$.
The verification of the initial conditions follows by noting that $y^n(0)\rightarrow y(0)$ in $H^2 \cap H_0^1(0,\pi)$, $\dot{y}^n(0)\rightarrow \dot{y}(0)$ in $L^2(0,\pi)$,
$\theta^n(0)\rightarrow \theta(0)$ in $H_0^1(0,\pi)$ and $\dot{\theta}^n(0)\rightarrow \dot{\theta}(0)$ in $L^2(0,\pi)$. The proof of the existence part is complete, once we
observe that all the above results hold for any $T>0$.\par
Next we turn to the uniqueness issue. Since it follows by repeating the proof of \cite[Theorem 11]{holubova} with some minor changes we only give a sketch of it.
Assume problem \eq{system}-\eq{boundary}-\eq{initial} admits two couples of solutions $(y^1,\theta^1)$ and $(y^2,\theta^2)$ and denote $(\bar y,\bar \theta):=(y^1-y^2,\theta^1-\theta^2)$.
Next, we
put $\mu_s(t)=-\int_t^s \bar y(\tau)\,d\tau$, $\eta_s(t)=-\int_t^s \bar \theta(\tau)\,d\tau$, $Y(t)=\int_0^t \bar y(\tau)\,d\tau$ and
$\Theta(t)=\int_0^t \bar \theta(\tau)\,d\tau$ with $0<t\leq s$. Note that $\mu_s'(t)=\bar y(t)$, $\eta_s'(t)=\bar \theta(t)$, $\mu_s(t)=Y(t)-Y(s)$ and $\eta_s(t)=\Theta(t)-\Theta(s)$.
Multiply the equation satisfied by $\bar y$ times $\mu_s$ and the one satisfied by $\bar \theta$ times $\eta_s$. By integrating, one deduces
$$
\left \{ \begin{array}{llll}
\displaystyle{\frac{1}{2}\|\bar y(s)\|_{2}^2+\frac{1}{2}\|Y_{xx}(s)\|_2^2}=\\
\int_0^s \left[f(y^1-\ell\sin \theta^1)+f(y^1+\ell\sin \theta^1)-f(y^2-\ell\sin \theta^2)-f(y^2+\ell\sin \theta^2)\right]\mu_sdt\\
\displaystyle{\frac{\ell}{6}\|\bar \theta(s)\|_{2}^2+\frac{\ell}{6}\|\Theta_x(s)\|_2^2}=\\
\int_0^s \left[\cos(\theta_1)(f(y^1-\ell\sin \theta^1)-f(y^1+\ell\sin \theta^1))-\cos(\theta_2)(f(y^2-\ell\sin \theta^2)-f(y^2+\ell\sin \theta^2))\right]\eta_sdt.
\end{array}\right.
$$
Exploiting the fact that $f\in$ Lip$_{loc}(\R)$ and \eq{uniform bound}, one infers that
$$\|\bar y(s)\|_{2}^2+\|\bar \theta(s)\|_{2}^2+\|Y_{xx}(s)\|_2^2+\|\Theta_x(s)\|_2^2\leq
C \int_0^s\left(\|\bar y(t)\|_{2}^2+\|\bar \theta(t)\|_{2}^2+\|Y_{xx}(t)\|_2^2+\|\Theta_x(t)\|_2^2\right)\,dt$$
where both the Young and the Poincar\'e inequalities have been exploited. Hence, by the Gronwall Lemma, $\|\bar y(s)\|_{2}=\|\bar \theta(s)\|_{2}=0$ and uniqueness follows.

\section{Proof of Theorem \ref{epsneg}}\label{proofepsneg}

We start with a calculus lemma.

\begin{lemma}\label{maxcalc}
Let $a,b,c,d>0$ and let $A=\{(x,y)\in\R^2;\, 0<y<dx\, ,\ ax+y+by^2<c\}$. Let
$$K_1=\frac{2c^2d}{(a+d)^2+2bcd^2+(a+d)\sqrt{(a+d)^2+4bcd^2}}>0\, ,$$
$$K_2=\frac{2(1+3bc)^{3/2}-2-9bc}{27ab^2}>0\, .$$
Then
$$\sup_{(x,y)\in A}\ xy\ =\max\{K_1,K_2\}\, .$$
\end{lemma}
\begin{proof} The set $A$ is delimited by two segments and an arch of parabola, see Figure \ref{figura}.
\begin{figure}[ht]
\begin{center}
{\includegraphics[height=28mm, width=70mm]{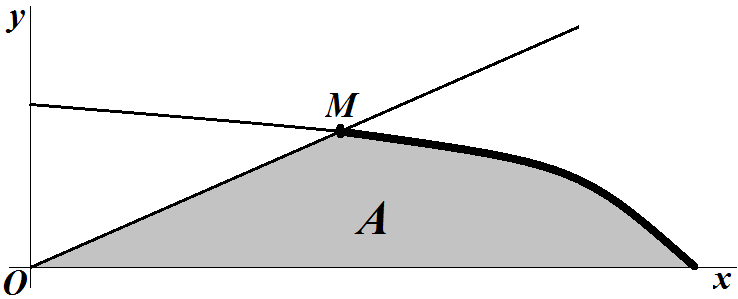}}
\caption{The set $A$.}\label{figura}
\end{center}
\end{figure}
In order to find the supremum of $xy$, we intersect the hyperbola $xy=k$ with the closure of $A$ and the maximum of $xy$ on such set belongs to the bold
face part of the boundary in Figure \ref{figura}. Whence, either it coincides with the point $M$ or with a point where a hyperbola $xy=k>0$ is tangent
to the parabola.\par
The coordinates of the point $M$ are given by
$$M\left(\frac{\sqrt{(a+d)^2+4bcd^2}-(a+d)}{2bd^2},\frac{\sqrt{(a+d)^2+4bcd^2}-(a+d)}{2bd}\right)$$
which belongs to the hyperbola $xy=K_1$.\par
The hyperbola tangent to the parabola is found by solving the following system:
$$xy=k\, ,\quad ax+y+by^2=c\, .$$
We have to determine for which $k$ this system admits a double solution. We find $k=K_2$.\par
Summarizing, the hyperbola having equation $xy=k$ intersect the bold face of the boundary in Figure \ref{figura} for $k=K_1$ if the intersection is
at the point $M$ and for $k=K_2$ if the hyperbola is tangent to the parabola. This proves the statement.\end{proof}

Then we recall a Gagliardo-Nirenberg \cite{gagliardo,nirenberg} inequality. Since we are interested in the value of the constant and since we were unable
to find one in literature, we also give its proof. We do not know if the optimal constant is indeed 1.

\begin{lemma}\label{symm}
For all $u\in H^1_0(0,\pi)$ we have $\|u\|_\infty^2\le\|u\|_2\, \|u'\|_2$.
\end{lemma}
{\it Proof. } Since symmetrization leaves $L^q$-norms of functions invariant and decreases the $L^p$-norms of the derivatives, see e.g.\
\cite[Theorem 2.7]{AlmgrenLieb}, we may restrict our attention to functions which are symmetric, positive and decreasing with respect to the
center of the interval. If $u$ is one such function we have
$$\int_0^{\pi/2}u(\tau)u'(\tau)\, d\tau=\int_0^{\pi/2}|u(\tau)u'(\tau)|\, d\tau=\int_{\pi/2}^\pi|u(\tau)u'(\tau)|\, d\tau=\frac12
\int_0^\pi|u(\tau)u'(\tau)|\, d\tau\, .$$
Therefore, by the H\"older inequality we have
$$\|u\|_\infty^2=u\left(\frac{\pi}{2}\right)^2=\int_0^{\pi/2}[u(\tau)^2]'\, d\tau=2\int_0^{\pi/2}u(\tau)u'(\tau)\, d\tau=
\int_0^\pi|u(\tau)u'(\tau)|\, d\tau\le\|u\|_2\, \|u'\|_2\, .\eqno{\Box}$$

In particular, by Lemma \ref{symm}, we have
\neweq{gagnir}
\|w^m(t)\|_\infty^4\le\|w^m(t)\|_2^2\, \|w^m_x(t)\|_2^2\, .
\endeq

Note also that an improved Poincar\'e inequality yields
\neweq{poincare}
\|w^m(t)\|_2\le\frac{1}{m+1}\|w^m_x(t)\|_2\, .
\endeq

Moreover, by \eq{split}, we may also rewrite \eq{Entot} as
\begin{eqnarray}
E &=& \frac{\|y_t(t)\|_2^2}{2}+\frac{\|z^m_t(t)\|_2^2}{6}+\frac{\|w^m_t(t)\|_2^2}{6}+
\frac{\|y_{xx}(t)\|_2^2}{2}+\frac{\|z^m_x(t)\|_2^2}{6}+\frac{\|w^m_x(t)\|_2^2}{6} \notag\\
\ & \ & +\|y(t)\|_2^2+\|z^m(t)\|_2^2+\|w^m(t)\|_2^2+\frac{\|y(t)\|_4^4}{2}+\frac{\|z(t)\|_4^4}{2}+3\int_0^\pi y(x,t)^2z(x,t)^2\, dx \notag\\
\mbox{(H\"older) } &> & \frac{\|w^m_x(t)\|_2^2}{6}+\|w^m(t)\|_2^2+\frac{\|z(t)\|_2^4}{2\pi} \notag\\
\mbox{by \eq{split} } &> & \frac{\|w^m_x(t)\|_2^2}{6}+\|w^m(t)\|_2^2+\frac{\|w^m(t)\|_2^4}{2\pi} \, .\label{E1}
\end{eqnarray}

The two constraints \eq{poincare} and \eq{E1} are as in Lemma \ref{maxcalc} with $x=\|w^m_x(t)\|_2^2$, $y=\|w^m(t)\|_2^2$, and
$$a=\frac16 \, ,\quad b=\frac{1}{2\pi}\, ,\quad c=E\, ,\quad d=\frac{1}{(m+1)^2}\, .$$
Therefore, Lemma \ref{maxcalc} combined with \eq{gagnir} states that
\neweq{stimawm}
\|w^m(t)\|_\infty^4\le\max\{K_1,K_2\}
\endeq
where
$$K_1=\frac{72\, \pi\, E^2\, (m+1)^2}{\ \pi(m^2+2m+7)^2+36E+(m^2+2m+7)\sqrt{\pi^2(m^2+2m+7)^2+72\pi E}\ }\, ,$$
$$K_2=\frac{16\, \pi^2}{9}\left[\left(1+\frac{3E}{2\pi}\right)^{3/2}-1\right]-4\, \pi\, E\, .$$

Then \eq{stimawm} yields $\|w^m(t)\|_\infty\le\omega$, provided either $K_1\le\omega^4$ or $K_2\le\omega^4$. The first occurs whenever
\eq{prima} holds whereas the second case occurs whenever \eq{seconda} holds.

\section{Proof of Theorem \ref{stable}}\label{5}

For any $E>0$ we put
$$\Lambda_{\pm}(E):=2\sqrt{1+\frac{2}{3}E}\pm 2.$$
Then we prove

\begin{lemma}\label{periodicT}
For any $\alpha,\beta\in\R$ problem \eqref{soloy} admits a unique solution $\overline{y}$ which is periodic of period
\neweq{TE}
T(E)=\frac{8}{\sqrt3 }\int_0^1\frac{ds}{\sqrt{(\Lambda_+(E)+\Lambda_-(E)s^2)(1-s^2)}}\, .
\endeq
In particular, the map $E\mapsto T(E)$ is strictly decreasing and $\lim_{E\to0}T(E)=2\pi/\sqrt3$.
\end{lemma}
\begin{proof} The existence, uniqueness and periodicity of the solution $\overline{y}$ is a known fact from the theory of ODE's. For a given $E>0$,
we may rewrite \eq{energyy} as
\neweq{ancoraenergia}
\dot{y}^2=2E-3y^2-\frac34 y^4\,.
\endeq
Hence,
\begin{equation}\label{y1sup}
\|\overline{y}\|_\infty=\sqrt{{\Lambda_-(E)}}\,.
\end{equation}
Since \eq{soloy} merely consists of odd terms, the period $T(E)$ of $\overline{y}$ is the double of the width of an interval of monotonicity for $\overline{y}$.
Since the problem is autonomous, we may assume that $\overline{y}(0)=-\|\overline{y}\|_\infty$ and $\dot{\overline{y}}(0)=0$; then, by symmetry and periodicity,
we have that
$\overline{y}(T/2)=\|\overline{y}\|_\infty$ and $\dot{\overline{y}}(T/2)=0$. By rewriting \eq{ancoraenergia} as
$$\dot{y}=\frac{\sqrt{3}}{2}\sqrt{(\Lambda_+(E)+y^2)(\Lambda_-(E)-y^2)}\qquad\forall t\in\left(0,\frac{T}{2}\right)\, ,$$
by separating variables, and upon integration over the time interval $(0,T/2)$ we obtain
$$\frac{T(E)}{2}= \frac{2}{\sqrt{3}} \int_{-\|\overline{y}\|_\infty}^{\|\overline{y}\|_\infty}\frac{dy}{\sqrt{(\Lambda_+(E)+y^2)(\Lambda_-(E)-y^2)}}\,.$$
Then, using the fact that the integrand is even with respect to $y$ and through a change of variable,
$$T(E)=\frac{8}{\sqrt{3}}\int_0^{\|\overline{y}\|_\infty}\frac{dy}{\sqrt{(\Lambda_+(E)+y^2)(\|\overline{y}\|_\infty^2-y^2)}}=\frac{8}{\sqrt3 }\int_0^1\frac{ds}{\sqrt{(\Lambda_+(E)+\Lambda_-(E)s^2)(1-s^2)}}\,,$$
which proves \eq{TE}. Both the maps $E \mapsto \Lambda_\pm(E)$ are continuous and increasing for $E\in[0,\infty)$ and $\Lambda_-(0)=0$, $\Lambda_+(0)=4$.
Whence, $E\mapsto T(E)$ is strictly decreasing and
$$\lim_{E\to0}T(E)=T(0)=\frac{4}{\sqrt3 }\int_0^1\frac{ds}{\sqrt{1-s^2}}=\frac{2\pi}{\sqrt3}\, ,$$
a result that could have also been obtained by noticing that, as $E\to0$, the equation \eq{soloy} tends to $\ddot{y}+3y=0$.\end{proof}

In the sequel, we need bounds for $T(E)$. From \eq{TE} we see that, by taking $s=0$ in the first polynomial under square root,
\neweq{upper}
T(E)\le\frac{8}{\sqrt{3\Lambda_+(E)}}\int_0^1\frac{ds}{\sqrt{1-s^2}}=\frac{4\pi}{\sqrt{3\Lambda_+(E)}}\ \Longrightarrow\ \frac{16\pi^2}{T(E)^2}\ge3\Lambda_+(E)\, .
\endeq
Moreover, by taking $s=1$ in the first polynomial under square root, we infer
\neweq{lower}
T(E)\ge\frac{8}{\sqrt{3}\sqrt{\Lambda_+(E)+\Lambda_-(E)}}\int_0^1\frac{ds}{\sqrt{1-s^2}}=\frac{2\pi}{\sqrt[4]{9+6E}}\ \Longrightarrow\
\frac{4\pi^2}{T(E)^2}\le\sqrt{9+6E}\, .
\endeq

Let us now consider \eq{nonlinearhill}. With the initial conditions $\xi(0)=\dot{\xi}(0)=0$, the unique solution of \eq{nonlinearhill} is $\xi\equiv0$. We are
interested in determining whether the trivial solution is stable in the Lyapunov sense, namely if the solutions of \eq{nonlinearhill} with small initial
data $|\xi(0)|$ and $|\dot{\xi}(0)|$ remain small for all $t\ge0$. By Lemma \ref{periodicT}, the function $\overline{y}(t)^2$ is $T/2$-periodic.
Then $a$ is a positive $T/2$-periodic function and a stability criterion for the Hill equation due to Zhukovskii \cite{zhk}, see also
\cite[Chapter VIII]{yakubovich}, states that the trivial solution of \eq{nonlinearhill} is stable provided that
\begin{equation}\label{stabper}
\frac{4\pi^2}{T(E)^2}\leq a(t)\leq \frac{16\pi^2}{T(E)^2}\, .
\end{equation}
Let us translate this condition
in terms of $\|\overline{y}\|_\infty$. By the definition of $a$ in \eq{nonlinearhill} and by \eq{y1sup}, we have
$$
7\le a(t)\le7+\frac{27}{2}\|\overline{y}\|_\infty^2=-20+9\sqrt{9+6E}\, .
$$
Whence, \eq{stabper} holds if both
\neweq{piusuff}
\frac{4\pi^2}{T(E)^2}\leq7\quad \text{and}\quad -20+9\sqrt{9+6E}\leq \frac{16\pi^2}{T(E)^2}\, .
\endeq
In turn, by \eq{upper}-\eq{lower}, the inequalities in \eq{piusuff} certainly hold if
$$\sqrt{9+6E}\le7\quad \text{and}\quad  -20+9\sqrt{9+6E}\le3\Lambda_+(E)\ .$$
The first of such inequalities is fulfilled provided that $E\le\frac{20}{3}$; the second inequality is satisfied if
\neweq{suff}
E\le\frac{235}{294}\approx0.799\ ,\qquad \|\overline{y}\|_\infty\le\sqrt{\frac{10}{21}}\approx0.69\ ,
\endeq
which is more stringent and, therefore, yields a sufficient condition for \eq{stabper} to hold. This proves Theorem \ref{stable}.

\begin{remark} {\em The sufficient condition \eqref{stabper} is fulfilled as long as both \eqref{piusuff} hold.
Numerically, we see that the former is satisfied for $E\lessapprox10.445$ whereas the latter is satisfied for $E\lessapprox0.944$.
The most stringent is the second one which corresponds to $\|\overline{y}\|_\infty\lessapprox0.74$, not significantly better than \eqref{suff}.}\end{remark}

\section{Proof of Theorem \ref{stable2}}\label{22}

For any $E>0$, we put
$$\Lambda_{\pm}^j(E)=2\sqrt{\frac{(j^4+2)^2}{9}+\frac{2}{3}E}\, \pm\, \frac{2}{3}\, (j^4+2)\qquad(j=1,2)\,.$$
Then \eq{energyEj}, with $E_j=E$, reads
$$\dot{y_j}^2=\frac{3}{4}(\Lambda_{+}^j(E)+y_j^2)(\Lambda_{-}^j(E)-y_j^2)\qquad (j=1,2)\,.$$
By this, since any $j$-th vertical mode $\overline{y}_j$ satisfies \eq{energyEj}, we deduce
\begin{equation}\label{yisupn}
\|\overline{y}_j\|_\infty=\sqrt{{\Lambda_{-}^j(E)}}\qquad (j=1,2)\,.
\end{equation}
Then, the same analysis performed in Section \ref{5} yields that the $\overline{y}_j$ are periodic functions of period
$$
T_j(E)=\frac{8}{\sqrt3 }\int_0^1\frac{ds}{\sqrt{(\Lambda_+^j(E)+\Lambda_-^j(E)s^2)(1-s^2)}}\, .
$$
In particular, the map $E\mapsto T_j(E)$ is strictly decreasing and $\lim_{E\to0}T_j(E)=2\pi/\sqrt{j^4+2}$. Furthermore, the following estimates hold
\neweq{upperin}
T_j(E)\le\frac{8}{\sqrt{3\Lambda^j_+(E)}}\int_0^1\frac{ds}{\sqrt{1-s^2}}=\frac{4\pi}{\sqrt{3\Lambda^j_+(E)}}\ \Longrightarrow\ \frac{16\pi^2}{T_j(E)^2}\ge3\Lambda^j_+(E)\,
\endeq
and
\neweq{lowerin}
T_j(E)\ge\frac{8}{\sqrt{3}\sqrt{\Lambda^j_+(E)+\Lambda^j_-(E)}}\int_0^1\frac{ds}{\sqrt{1-s^2}}=\frac{2\pi}{\sqrt[4]{(j^4+2)^2+6E}}\ \Longrightarrow\
\frac{4\pi^2}{T_j(E)^2}\le\sqrt{(j^4+2)^2+6E}\, .
\endeq
Consider the first mode $(Y_1,0)=(\overline{y}_1,0,0,0)$. For $j=1$ the system \eq{nonlinearhillsystem2} reads
\neweq{nonlinearhillsystem21}
\left \{ \begin{array}{ll}
\ddot{\xi_1}(t)+(7+\frac{27}{2}\overline{y}_1(t)^2)\xi_1(t)=0\\
\ddot{\xi_2}(t)+(10+9\overline{y}_1(t)^2)\xi_2(t)=0\,.
\end{array}\right.
\endeq
If the trivial solution of both the equations in \eq{nonlinearhillsystem21} is stable then system \eq{nonlinearhillsystem21} itself is stable and
Definition \ref{newdeff} is satisfied, see \cite[Theorem II-Chapter III-vol 1]{yakubovich}. Since the first equation in \eq{nonlinearhillsystem21}
coincides with \eq{nonlinearhill}, the proof of Theorem \ref{stable} yields the torsional stability provided that \eq{suff} holds. For the second equation in
\eq{nonlinearhillsystem21}, by applying again the Zhukovskii stability criterion \eq{stabper}, we see that the trivial solution is stable provided that
$$\frac{4\pi^2}{T(E)^2}\leq 10+9\overline{y}_1(t)^2\leq \frac{16\pi^2}{T(E)^2} \,.$$
By arguing as in the proof of Theorem \ref{stable}, see \eq{piusuff}, we reach the bounds \eq{bounds2m} which are more stringent than \eq{suff}. These are
the bounds for the torsional stability of the first vertical mode.\par
For the second vertical mode $(Y_2,0)=(0,\overline{y}_2,0,0)$ we proceed similarly, but now system \eq{nonlinearhillsystem2} reads
\neweq{nonlinearhillsystem22}
\left \{ \begin{array}{ll}
\ddot{\xi_1}(t)+(7+9\overline{y}_2(t)^2)\xi_1(t)=0\\
\ddot{\xi_2}(t)+(10+\frac{27}{2}\overline{y}_2(t)^2)\xi_2(t)=0\,.
\end{array}\right.
\endeq
Concerning the first equation, a different stability criterion for the Hill equation due to Zhukovskii \cite{zhk}, see also \cite[Chapter VIII]{yakubovich},
states that the trivial solution is stable provided that
\neweq{ancoraa}
0\leq 7+9\overline{y}_2(t)^2 \leq \frac{4\pi^2}{T_2(E)^2}\, .
\endeq
The left inequality is always satisfied while the second inequality is satisfied if $7+9\|\overline{y}_2\|_\infty^2\le \frac{4\pi^2}{T_2(E)^2}$.
Whence, by \eq{upperin} a sufficient condition for the stability is
\neweq{less}
7+9\|\overline{y}_2\|_\infty^2\le\frac34 \Lambda_+^2(E)\ \Longleftrightarrow\ E\le\frac{38}{3}\,,\ \|\overline{y}_2\|_\infty\le\frac{2}{\sqrt3 }\ .
\endeq

Next we focus on the second equation in \eq{nonlinearhillsystem22}. The stability of the trivial solution is ensured if
$0\leq 10+\frac{27}{2}\overline{y}_2(t)^2 \leq \frac{4\pi^2}{T_2(E)^2}$, that is, if $10+\frac{27}{2}\|\overline{y}_2\|_\infty^2\le \frac{4\pi^2}{T_2(E)^2}$.
Whence, by \eq{upperin} with $j=2$, a sufficient condition for the stability is
\neweq{less2}
10+\frac{27}{2}\|\overline{y}_2\|_\infty^2\le\frac34 \Lambda_+^2(E)\ \Longleftrightarrow\ E\le\frac{5024}{867}\,,\
\|\overline{y}_2\|_\infty\le\sqrt{\frac{32}{51}}\ .
\endeq
This is more restrictive than \eq{less} and is therefore a sufficient condition for the stability of the second vertical mode $\overline{y}_2$.
The proof of Theorem \ref{stable2} is now complete.

\section{Appendix}\label{calculus lemma}

In order to determine the projected system, we need to multiply by $\sin(jx)$ ($j\in\N$) the equations in \eq{fcon1} and then integrate over $(0,\pi)$.
Therefore, we intensively exploit the following\renewcommand{\arraystretch}{1.1}

\begin{lemma}\label{calculus}
For all $k\in\N$ we have
$$c_{k,k,k}=\frac{8}{\pi}\int_0^\pi\sin^4(kx)\, dx=3\, .$$
For all $l,k\in\N$ ($l\neq k$) we have
$$
c_{l,k,k}=\frac{8}{\pi}\int_0^\pi\sin^3(kx)\sin(lx)\, dx=\left\{\begin{array}{ll}-1\ &\mbox{if }l=3k\\
0\ &\mbox{if }l\neq3k\, .
\end{array}\right.
$$
For all $l,k\in\N$ ($l\neq k$) we have
$$
c_{l,l,k}=\frac{8}{\pi}\int_0^\pi\sin^2(kx)\sin^2(lx)\, dx=2\, .
$$
For all $l,j,k\in\N$ (all different and $l<j$) we have
$$
c_{l,j,k}=\frac{8}{\pi}\int_0^\pi\sin^2(kx)\sin(jx)\sin(lx)\, dx=\left\{\begin{array}{ll}1\ &\mbox{if }j+l=2k\\
-1\ &\mbox{if }j-l=2k\\
0\ &\mbox{if }j\pm l\neq2k\, .
\end{array}\right.
$$
\end{lemma}
\begin{proof} By linearization we know that\vskip-1mm
$$\sin^4\theta=\frac{\cos(4\theta)-4\cos(2\theta)+3}{8}\, ,\quad\sin^3\theta=\frac{3\sin\theta-\sin(3\theta)}{4}\qquad\forall\theta\in\R\, .$$\vskip-1mm
Therefore, we may readily compute the two integrals\vskip-1mm
$$\int_0^\pi\sin^4(kx)\, dx=\int_0^\pi\frac{\cos(4kx)-4\cos(2kx)+3}{8}\, dx=\frac38\, \pi\, ,$$\vskip-1mm
\begin{eqnarray*}
\int_0^\pi\sin^3(kx)\sin(lx)\, dx &=& \int_0^\pi\sin(lx)\frac{3\sin(kx)-\sin(3kx)}{4}\, dx\\
\ &=&-\frac14 \int_0^\pi\sin(lx)\sin(3kx)\, dx=\left\{\begin{array}{ll}-\frac{\pi}8\ &\mbox{if }l=3k\\
0\ &\mbox{if }l\neq3k\, .
\end{array}\right.
\end{eqnarray*}
Next, we compute\vskip-1mm
$$
\int_0^\pi\sin^2(kx)\sin^2(lx)\, dx=\frac14 \int_0^\pi(1-\cos(2kx))(1-\cos(2lx))\, dx=\frac{\pi}{4}\, .
$$\vskip-1mm
Finally, by using the prostaphaeresis formula, we get\vskip-1mm
$$
\int_0^\pi\sin^2(kx)\sin(jx)\sin(lx)\, dx = \frac12 \int_0^\pi(1-\cos(2kx))\sin(jx)\sin(lx)\, dx
$$
\vskip-5mm
$$
=\frac14 \int_0^\pi\cos(2kx)[\cos((j-l)x)+\cos((j+l)x)]\, dx=
\left\{\begin{array}{ll}\frac{\pi}8\ &\mbox{if }j+l=2k\\
-\frac{\pi}8\ &\mbox{if }j-l=2k\\
0\ &\mbox{if }j\pm l\neq2k\, .
\end{array}\right.
$$\vskip-1mm
The proof of the four identities is so complete.\end{proof}

Lemma \ref{calculus} allows to compute the coefficients of problems \eq{inf syst2} and \eq{systemU}. For the latter, one needs to compute coefficients of the kind
$c_{l,k,k}$. If $m\ge3$, since these coefficients do not vanish when $l=3k$, \eq{systemU} with initial conditions \eq{infork} may admit solutions $(Y,0)$ which have
more than one nontrivial component within the vector $Y$. Then the study of the stability of the $k$-th vertical mode, see Definition \ref{oscillmode},
cannot be reduced to that of a single Hill equation as in the cases $m=1$ and $m=2$. One obtains, instead, a system of coupled equations. The problem of the stability
for general $m\ge3$ will be addressed in \cite{bergazpie}.

\end{document}